\documentclass[11pt,english]{extarticle}

\usepackage[latin9]{inputenc}
\usepackage{geometry}
\usepackage{color}
\usepackage{babel}
\usepackage{mathrsfs}
\usepackage{amsmath}
\usepackage{amsthm}
\usepackage{amssymb}
\usepackage[pdfusetitle,
 bookmarks=true,bookmarksnumbered=false,bookmarksopen=false,
 breaklinks=false,pdfborder={0 0 1},backref=false,colorlinks=true]
 {hyperref}
\hypersetup{
 pdfborderstyle=,linkcolor=blue,urlcolor=cyan,citecolor=red}

\makeatletter
\numberwithin{equation}{section}
\numberwithin{figure}{section}
\numberwithin{table}{section}
\usepackage{paralist}
\theoremstyle{plain}
\newtheorem{thm}{\protect\theoremname}[section]
\theoremstyle{plain}
\newtheorem{lyxalgorithm}[thm]{\protect\algorithmname}
\theoremstyle{definition}
\newtheorem{defn}[thm]{\protect\definitionname}
\theoremstyle{definition}
\newtheorem{rem}[thm]{\protect\remarkname}
\theoremstyle{plain}
\newtheorem{prop}[thm]{\protect\propositionname}
\theoremstyle{plain}
\newtheorem{cor}[thm]{\protect\corollaryname}
\theoremstyle{definition}
\newtheorem{example}[thm]{\protect\examplename}
\theoremstyle{plain}
\newtheorem{lem}[thm]{\protect\lemmaname}

\@ifundefined{date}{}{\date{}}
\usepackage{babel}
\usepackage{babel}
\usepackage{babel}
\usepackage{babel}
\usepackage{babel}
\usepackage{babel}
\usepackage{babel}
\usepackage{babel}
\usepackage{babel}
\usepackage{babel}
\usepackage{xpatch}
\usepackage{babel}
\usepackage{babel}
\usepackage{babel}

\usepackage[misc]{ifsym}

\usepackage{color}
\def\mystrut(#1,#2){\vrule height #1pt depth #2pt width 0pt}
\usepackage[normalem]{ulem}
\newcommand{\culine}{\bgroup\markoverwith
{\textcolor{red}{\rule[-0.5ex]{2pt}{0.5pt}}}\ULon}

\def\cdashuline{\bgroup
\UL@setULdepth
\markoverwith{\textcolor{red}{\kern.13em
\vtop{\kern\ULdepth \hrule width .3em}%
\kern.13em}}\ULon}
\def\cuuline{\bgroup \UL@setULdepth
\markoverwith{\textcolor{red}{\lower\ULdepth\hbox
{\kern-.03em\vbox{\hrule width.2em\kern1.2\p@\hrule}\kern-.03em}}}
\ULon}
\usepackage{enumitem}
\setlist[itemize,1]{label={\fontfamily{cmr}\fontencoding{T1}\selectfont\textbullet}}

\renewenvironment{proof}[1][\proofname] {\par\pushQED{\qed}\normalfont\topsep6\p@\@plus6\p@\relax\trivlist\item[\hskip\labelsep\bfseries#1\@addpunct{.}]\ignorespaces}{\popQED\endtrivlist\@endpefalse}

\AtBeginDocument{\xpatchcmd{\@thm}{\thm@headpunct{.}}{\thm@headpunct{.}}{}{}}

\usepackage{mdwlist}

\providecommand{\definitionname}{Definition}
\providecommand{\examplename}{Example}
\providecommand{\lemmaname}{Lemma}
\providecommand{\propositionname}{Proposition}
\providecommand{\remarkname}{Remark}
\providecommand{\theoremname}{Theorem}

\providecommand{\definitionname}{Definition}
\providecommand{\examplename}{Example}
\providecommand{\lemmaname}{Lemma}
\providecommand{\propositionname}{Proposition}
\providecommand{\remarkname}{Remark}
\providecommand{\theoremname}{Theorem}

\providecommand{\definitionname}{Definition}
\providecommand{\examplename}{Example}
\providecommand{\lemmaname}{Lemma}
\providecommand{\propositionname}{Proposition}
\providecommand{\remarkname}{Remark}
\providecommand{\theoremname}{Theorem}

\providecommand{\definitionname}{Definition}
\providecommand{\examplename}{Example}
\providecommand{\lemmaname}{Lemma}
\providecommand{\propositionname}{Proposition}
\providecommand{\remarkname}{Remark}
\providecommand{\theoremname}{Theorem}

\providecommand{\corollaryname}{Corollary}
\providecommand{\definitionname}{Definition}
\providecommand{\examplename}{Example}
\providecommand{\lemmaname}{Lemma}
\providecommand{\propositionname}{Proposition}
\providecommand{\remarkname}{Remark}
\providecommand{\theoremname}{Theorem}

\providecommand{\lemmaname}{Lemma}
\providecommand{\propositionname}{Proposition}
\providecommand{\remarkname}{Remark}
\providecommand{\theoremname}{Theorem}

\providecommand{\lemmaname}{Lemma}
\providecommand{\propositionname}{Proposition}
\providecommand{\remarkname}{Remark}
\providecommand{\theoremname}{Theorem}

\providecommand{\lemmaname}{Lemma}
\providecommand{\propositionname}{Proposition}
\providecommand{\remarkname}{Remark}
\providecommand{\theoremname}{Theorem}

\providecommand{\lemmaname}{Lemma}
\providecommand{\remarkname}{Remark}
\providecommand{\theoremname}{Theorem}

\providecommand{\lemmaname}{Lemma}
\providecommand{\remarkname}{Remark}
\providecommand{\theoremname}{Theorem}

\providecommand{\lemmaname}{Lemma}
\providecommand{\remarkname}{Remark}
\providecommand{\theoremname}{Theorem}

\providecommand{\corollaryname}{Corollary}
\providecommand{\definitionname}{Definition}
\providecommand{\examplename}{Example}
\providecommand{\lemmaname}{Lemma}
\providecommand{\remarkname}{Remark}
\providecommand{\theoremname}{Theorem}

\providecommand{\corollaryname}{Corollary}
\providecommand{\definitionname}{Definition}
\providecommand{\examplename}{Example}
\providecommand{\lemmaname}{Lemma}
\providecommand{\remarkname}{Remark}
\providecommand{\theoremname}{Theorem}
\renewcommand\theenumi{(\roman{enumi})}
\usepackage{microtype}
\usepackage{indentfirst}
\DeclareMathOperator{\im}{Im}

\makeatother

\providecommand{\algorithmname}{Algorithm}
\providecommand{\corollaryname}{Corollary}
\providecommand{\definitionname}{Definition}
\providecommand{\examplename}{Example}
\providecommand{\lemmaname}{Lemma}
\providecommand{\propositionname}{Proposition}
\providecommand{\remarkname}{Remark}
\providecommand{\theoremname}{Theorem}

\begin{document}
\title{Strong convergence, perturbation resilience and superiorization of
Generalized Modular String-Averaging with infinitely many input operators}
\author{Kay Barshad and Yair Censor\\
Department of Mathematics, University of Haifa, Mt. Carmel, \\
Haifa 3498838, Israel \\
\Letter ~ \href{mailto:kaybarshad@gmail.com}{kaybarshad@gmail.com};
\Letter ~ \href{mailto:yair@math.haifa.ac.il}{yair@math.haifa.ac.il}}
\date{September 15, 2025. Revised: December 23, 2025.}
\maketitle
\begin{abstract}
We study the strong convergence and bounded perturbation resilience
of iterative algorithms based on the Generalized Modular String-Averaging
(GMSA) procedure for infinite sequences of input operators under a
general admissible control. These methods address a variety of feasibility-seeking
problems in real Hilbert spaces, including the common fixed point
problem and the convex feasibility problem. In addition to the general
case, involving certain strongly quasi-nonexpansive input operators,
we consider a specific subclass of their corresponding relaxed firmly
nonexpansive operators. This subclass proves useful for establishing
bounded perturbation resilience. We further demonstrate the applicability
of our strong convergence results, within the GMSA framework, to the
Superiorization Methodology and to Dynamic String-Averaging, analyzing
the behavior of a superiorized version of our main algorithm. The
novelty and significance of this work is that it not only includes
a variety of earlier algorithms as special cases but, more importantly,
it allows the use of modular options of string-averaging that give
rise to new, hitherto unavailable, algorithmic schemes with emphasis
on infinitely many input operators. The strong convergence guarantees
and the applications for superiorization and dynamic string-averaging
are also important facets.
\end{abstract}
\renewcommand\theenumi{(\roman{enumi})}


\section{Introduction\protect\label{sec:Introduction}}

The Modular String-Averaging (MSA) procedure for a finite number of
input operators was developed by Reich and Zalas in \cite{MSA}, based
on the original string-averaging scheme introduced by \cite{CEH2001},
and on the work in \cite{CensorTom2003}, to provide a flexible algorithmic
framework for a large family of iterative methods. The MSA thus not
only includes a variety of earlier algorithms as special cases but,
more importantly, it allows the use of new modular options of string-averaging
that give rise to new, hitherto not available, algorithmic schemes.

On the other hand, Bauschke and Combettes suggested in \cite{BC}
the notion of coherence which turned out to be an important approach
in proving weak and strong convergence of such iterative processes.
The stronger version of coherence, was presented by Barshad, Reich
and Zalas in \cite{BRZ(R)}, as a useful tool for establishing the
coherence of a sequence of certain operators.

In 2023 Barshad, Gibali and Reich combined the ideas of both MSA and
strong coherence, by proposing in \cite{BGR2} the Generalized Modular
String-Averaging (GMSA) procedure for an infinite number of input
operators and showing the strong coherence of its output operators,
where the general admissible control was used. In that work, the weak
convergence of GMSA methods was proved based on the theory of coherence,
while the strong convergence necessitated the use of Haugazeau projections.

In many problems the computation of output operators can be inexact
and produce some small errors which we call perturbations. Unfortunately,
it is difficult and not convenient to consider perturbations by using
Haugazeau projection methods when we are interested in the strong
convergence of iterative algorithms in the infinite-dimensional Hilbert
space. For this reason formulating alternative conditions on the input
operators under which the strong convergence will hold in the infinite-dimensional
space is desirable. Such conditions were provided in \cite{MSA},
where the properties of the MSA procedure for a finite family of input
operators under an $s$-intermittent control were studied.

In this paper we develop extensions of these conditions to a much
more general method, which involves relaxation parameters, based on
the GMSA procedure for an infinite family of input operators under
a more general admissible control, where we additionally show its
bounded perturbation resilience. Of particular interest is the case
of relaxed firmly nonexpansive operators, where we can guarantee the
bounded perturbation resilience of our strongly convergent methods.
This enables us to study the properties of the superiorized version
of our GMSA algorithm.

The situation of an infinite family of input operators is mathematically
interesting but also has its roots in modeling of real-world problems.
Combettes \cite[1997]{PLC} noted that finding a common point of a
family of closed and convex sets in a Hilbert space, known as the
hilbertian convex feasibility problem, captures problems in disciplines
as diverse as approximation theory, integral equations, control theory,
signal and image processing, biomedical engineering, communications,
and geophysics. The ``extrapolated method of parallel projections''
(EMOPP), which allows the total number of sets to be countably infinite,
was proposed in that paper. Kong, Pajoohesh and Herman \cite{KPH}
study string-averaging algorithms for convex feasibility with infinitely
many sets and make a convincing case, backed by many references, for
the importance of infinitely many sets. The latter translate to infinitely
many input operators in our framework, see also \cite{AJZ1} and \cite{AJZ}.

Since its inception in 2007, the superiorization method (SM) has evolved
and gained ground. Recent review papers on the subject are Herman's
\cite{Herman2014} and \cite{Herman2020} and the review in \cite{asymmetric-2023}.
The superiorization method was born when the terms and notions \textquotedblleft superiorization\textquotedblright{}
and \textquotedblleft perturbation resilience\textquotedblright ,
in the present context, first appeared in the 2009 paper \cite{DavidiHermanCensor2009}
which followed its 2007 forerunner by Butnariu et al. \cite{ButnariuDavidiHermanKazantsev2007}.
The ideas have some of their roots in the 2006 and 2008 papers of
Butnariu et al. \cite{BuRZ1,BuRZ}, where it was shown that if iterates
of a nonexpansive operator converge for any initial point, then its
inexact iterates with summable errors also converge.

Bounded perturbation resilience of a parallel projection method was
observed as early as 2001 in \cite[Theorem 2]{Combettes_2001} (without
using this term). All these culminated in Ran Davidi\textquoteright s
2010 PhD dissertation \cite{Davidi2010} and the many papers that
appeared since then and are cited in \cite{SM-bib-page} which is
a Webpage dedicated to superiorization and perturbation resilience
of algorithms that contains a continuously updated bibliography on
the subject. This Webpage\footnote{http://math.haifa.ac.il/yair/bib-superiorization-censor.html\#top,
last updated on February 22, 2026, with 204 items} is source for the wealth of work done in this field to date, including
two special issues of journals \cite{CHJ-special-issue-2017} and
\cite{GibaliHermanSchnoerr2020} dedicated to research of the SM.
Interestingly, \cite{AndersenHansen2014} notices some structural
similarities of the SM with incremental proximal gradient methods.

The 2001 paper of Combettes \cite{Combettes2001} also investigated
the use summable perturbations on cutter-type methods. For example,
Theorem 5.2 of \cite{Combettes2001} is a bounded perturbation resilience
result for Algorithm 5.1 there, without using this term. Many more
results in \cite{Combettes2001} are related directly to the SM, without
using the language of the SM.

The ``adaptive steepest descent projections onto convex sets'' (ASD-POCS)
algorithm described in \cite{Sidky2008} and in many subsequent works
on it, has some similarities to the SM. However, it is not as general
as the SM; see \cite{Herman2012} for a comparison between it and
the SM.

The paper is laid out as follows. We set up notations and present
preliminaries in Section \ref{sec:Preliminaries} with emphasis on
various types of algorithmic operators in Subsection \ref{subsec:Various-types-of}.
General bounded regularity and approximate shrinking properties of
operators are discussed in Subsection \ref{subsec:The-general-bounded}
and some additional notions and results are presented in Subsection
\ref{subsec:Additional-notions-and}. In Subsection \ref{GMSA} the
general modular string-averaging procedure is recalled. In Section
\ref{sec:The-strong-convergence} we establish our results concerning
the strong convergence properties of the GMSA procedure based methods.
Namely, Subsection \ref{str_quas_non} describes results in the general
case of strongly quasi nonexpansive operators, while in Subsection
\ref{firm_nonex} we focus on the particular case of relaxed firmly
nonexpansive operators. In Section \ref{SM} we investigate certain
properties of the general superiorization algorithm. Finally, in the
last Section \ref{Apps} we present applications of our results to
the Superiorization Methodology and Dynamic String Averaging.

\section{\protect\label{sec:Preliminaries}Preliminaries}

Throughout this paper, $\mathbb{Z}$ denotes the set of integer numbers,
$\mathbb{N}$ denotes the set of natural numbers (including $0$),
and for any two integers $m$ and $n$, with $m\le n$, we denote
by $\left\{ m,m+1,\dots,n\right\} $ the set of all integers between
$m$ and $n$. For a set $A$, we denote by $\left|A\right|$ the
cardinality of $A$. For a real Hilbert space $\mathcal{H}$, we use
the following notations:
\begin{itemize}
\item $\langle\cdot,\cdot\rangle$ denotes the inner product on $\mathcal{H}.$
\item $\Vert\cdot\Vert$ denotes the norm on $\mathcal{H}$ induced by $\langle\cdot,\cdot\rangle.$
\item $Id$ denotes the identity operator on $\mathcal{H}$.
\item $\mathrm{Fix}T$ denotes the set $\mathrm{Fix}T:=\{x\in\mathcal{H}\mid T(x)=x\}$
of fixed points of an operator $T:\mathcal{H}\rightarrow\mathcal{H}$.
\item For a nonempty and convex subset $C$ of $\mathcal{H}$, we denote
by $P_{C}$ the (unique) metric projection onto $C$, the existence
of which is guaranteed if $C$ is, in addition, closed.
\item The expression $x^{k}\rightarrow x$ denotes the strong convergence
to $x$ of a sequence $\left\{ x^{k}\right\} _{k=0}^{\infty}$ in
$\mathscr{\left(\mathcal{H}\mathscr{,\Vert\cdot\Vert}\right)}$ when
$k\rightarrow\infty$.
\item For a convex function $\phi:\mathcal{H}\rightarrow\mathbb{R}$, where
$\mathbb{R}$ denotes the real line, and a point $x\in\mathcal{H}$,
we denote by $\partial\phi\left(x\right)$ the subdifferential set
of $\phi$ at $x$, that is,
\[
\partial\phi\left(x\right):=\left\{ g\in\mathcal{H}\,|\,\left\langle g,y-x\right\rangle \le\phi\left(y\right)-\phi\left(x\right)\,\,\mathrm{for}\,\,\mathrm{all}\,\,y\in\mathcal{H}\right\} .
\]
\item For a function $f:\mathcal{H}\rightarrow\mathbb{R}$ and a subset
$A$ of $\mathcal{H}$, we denote by $\underset{x\in A}{\mathrm{Argmin}}f\left(x\right)$
the set of minimizers of $f$ on the set $A$.
\item $B\left(x,r\right)$ denotes the open ball centered at $x\in\mathcal{H}$
of radius $r>0$.
\item For a nonempty subset $C$ of $\mathcal{H}$ and $x\in\mathcal{H}$,
we denote by $d\left(x,C\right)$ the distance from $x$ to $C$,
that is, $d\left(x,C\right):=\inf_{y\in C}\left\Vert x-y\right\Vert $.
\end{itemize}
We study the properties of the following algorithm, which serves as
the algorithmic framework for our GMSA methods.
\begin{lyxalgorithm}
[The algorithmic framework]\label{alg=0000201}

Given $\varepsilon\in\left(0,1\right]$, $x^{0}\in\mathcal{H}$ and
a sequence $\left\{ T_{k}\right\} _{k=0}^{\infty}$ of operators,
$T_{k}:\mathcal{H\rightarrow\mathcal{H}}$ for each $k\in\mathbb{N}$,
the algorithm is defined by the recurrence
\[
x^{k+1}:=x^{k}+\lambda_{k}\left(T_{k}\left(x^{k}\right)-x^{k}\right),
\]
where $\lambda_{k}\in\left[\varepsilon,2-\varepsilon\right]$ for
each $k\in\mathbb{N}$.
\end{lyxalgorithm}

\subsection{\protect\label{subsec:Various-types-of}Various types of algorithmic
operators and their properties}

We review the following classes of algorithmic operators together
with key results regarding their properties. For additional details,
see, for instance, \cite{C_book}.
\begin{defn}
Let $T:\mathcal{H}\rightarrow\mathscr{\mathcal{H}}$ be an operator
and let $\lambda\in\left[0,2\right]$. The operator $T_{\lambda}:\mathcal{H}\rightarrow\mathscr{\mathcal{H}}$
defined by $T_{\lambda}:=\left(1-\lambda\right)\mathrm{Id}+\lambda T$
is called a $\lambda$-$relaxation$ of the operator $T$. The operator
$T_{2}$ is called a \textit{reflection} of the operator $T$.
\end{defn}
\begin{rem}
\label{Relaxation=000020has=000020the=000020same=000020Fix} Clearly,
$\mathrm{Fix}T=\mathrm{Fix}T_{\lambda}$ for every operator $T:\mathcal{H}\rightarrow\mathcal{H}$
and every $\lambda\in\left(0,2\right]$.
\end{rem}
\begin{defn}
An operator $T:\mathcal{H}\rightarrow\mathcal{H}$ is \textit{nonexpansive}
if 
\[
\left(\forall x,y\in\mathcal{H}\right)\,\,\left\Vert T\left(x\right)-T\left(y\right)\right\Vert \leq\left\Vert x-y\right\Vert .
\]
\end{defn}
For every ordered pair $\left(x,y\right)\in\mathscr{\mathcal{H}}^{2}$,
we define the closed and convex set $H\left(x,y\right)$ by

\[
H\left(x,y\right):=\left\{ u\in\mathcal{H}\,\vert\,\langle u-y,x-y\rangle\le0\right\} .
\]

\begin{defn}
An operator $T:\mathscr{\mathcal{H}}\rightarrow\mathscr{\mathcal{H}}$
is called a \textit{cutter} if it satisfies 
\[
\textnormal{Fix}T\subseteq H\left(x,T\left(x\right)\right),\text{ }\forall x\in\mathcal{H},
\]
or, equivalently, if $\left\langle z-T(x),x-T(x)\right\rangle \leq0$
for each $z\in\mathrm{Fix}T$ and $x\in\mathcal{H}$. For $\lambda\in\left[0,2\right]$,
an operator $T:\mathcal{H}\rightarrow\mathcal{H}$ is a $\lambda$-\textit{relaxed
}cutter if $T$ is a $\lambda$-relaxation of a cutter $U$, that
is, $T=U_{\lambda}=\left(1-\lambda\right)Id+\lambda U$.
\end{defn}
\begin{prop}
[{\cite[Proposition 2.6]{BC}}]\label{cutter_fix}Let $T$ be a cutter.
Then $\mathrm{Fix}T=\cap_{x\in\mathcal{H}}H\left(x,T\left(x\right)\right)$
and hence $\mathrm{Fix}T$ is a closed and convex subset of $\mathcal{H}$,
as an intersection of half-spaces.
\end{prop}
\begin{rem}
\label{Rem_cut}The class of cutters was originally introduced by
Bauschke and Combettes in \cite{BC} under a different terminology,
where it was referred to as the ``class $\mathfrak{\mathfrak{T}}$''.
The term ``cutter'' was later proposed in \cite{cegielski-censor-2011}.
Other names are used in the literature for these operators, for instance,
``firmly quasi-nonexpansive'' (see, for example, Definition 4.1
and Proposition 4.2 in \cite{BC_book}), which also contains various
properties and examples of these operators. It is important to alert
the reader about some ambiguity in the literature regarding this term.
Definition 9.2 of \cite{cegielski-censor-2011} provides the original
definition of cutters in which they are defined without any condition
on the non-emptiness of their fixed points sets. In this sense cutters
are simply another name for the members of the original ``class $\mathfrak{\mathfrak{T}}$''
of Bauschke and Combettes in \cite{BC} which are also defined there
without such a condition. However, in Definition 2.1.30 in \cite{C_book}
the definition was modified to require that the fixed points sets
of the cutter operators be nonempty. This has led to some ambiguity,
as some later publications either include or do not include this non-emptiness
assumption.

We follow the original definition (Definition 9.2 of \cite{cegielski-censor-2011})
and explicitly assume the non-emptiness of the fixed point set of
a cutter only when required.
\end{rem}
\begin{defn}
We say that an operator $T:\mathcal{H}\rightarrow\mathcal{H}$ is:
\begin{enumerate}
\item \textit{Quasi-nonexpansive} if 
\[
\left(\forall x\in\mathcal{H}\right)\left(z\in\mathrm{Fix}T\right)\left\Vert T\left(x\right)-z\right\Vert \le\left\Vert x-z\right\Vert .
\]
\item \textit{$\rho$-strongly quasi-nonexpansive} for some $0\le\rho\in\mathbb{R}$
if
\begin{equation}
\left(\forall x\in\mathcal{H}\right)\left(z\in\mathrm{Fix}T\right)\left\Vert T\left(x\right)-z\right\Vert ^{2}\le\left\Vert x-z\right\Vert ^{2}-\rho\left\Vert T\left(x\right)-x\right\Vert ^{2}.\label{eq:-1-1}
\end{equation}
If $T$ satisfies \eqref{eq:-1-1} for some $\rho>0$, then it is
called strongly quasi-nonexpansive.
\end{enumerate}
\end{defn}
\begin{thm}
[{\cite[Theorem 2.1.39]{C_book}}]\label{thm:2.1.39}Let $T:\mathcal{H}\rightarrow\mathcal{H}$
be an operator and let $\lambda\in\left(0,2\right]$. Then $T$ is
a cutter if and only if its relaxation $T_{\lambda}$ is $\frac{2-\lambda}{\lambda}$-strongly
quasi-nonexpansive, that is, 
\[
\left\Vert T_{\lambda}x-z\right\Vert ^{2}\le\left\Vert x-z\right\Vert ^{2}-\frac{2-\lambda}{\lambda}\left\Vert T_{\lambda}x-x\right\Vert ^{2}
\]
for all $x\in\mathcal{H}$ and for all $z\in\mathrm{Fix}T$.
\end{thm}
\begin{thm}
[{\cite[Theorems 2.1.48 and 2.1.50]{C_book}}]\label{2.1.50}Let $m$
be a positive integer. For each \textup{$i=1,2,\dots,m$}, let $\rho_{i}>0$
and assume that $U_{i}:\mathscr{\mathcal{H}\rightarrow\mathcal{H}}$
is a $\rho_{i}$-strongly quasi-nonexpansive operator. Suppose that
$\bigcap_{i=1}^{m}\mathrm{Fix}U_{i}\not=\emptyset$. Define $\rho:=\min_{i\in\left\{ 1,2,\dots,m\right\} }\rho_{i}$.
Then the following two assertions hold:
\begin{enumerate}
\item The convex combination $U:=\sum_{i=1}^{m}\omega_{i}U_{i}$, where
$\omega_{i}\in\left(0,1\right]$ for each \textup{$i=1,2,\dots,m$}
, and $\sum_{i=1}^{m}\omega_{i}=1$, is $\rho$-strongly quasi-nonexpansive.
\item The composition $V:=U_{m}\cdots U_{2}U_{1}$ is $\rho m^{-1}$-strongly
quasi-nonexpansive.
\end{enumerate}
\end{thm}
\begin{rem}
In a manner similar to Remark \ref{Rem_cut} regarding cutters, we
note that in the literature (for instance, in Definitions 2.1.19 and
2.1.38 of \cite{C_book}) quasi-nonexpansive and strongly quasi-nonexpansive
operators are defined with the requirement that their fixed point
sets be nonempty. Here we define these operators without imposing
this condition, while we assume the non-emptiness of their fixed points
sets only if we need it.
\end{rem}
The following corollary is an immediate consequence of Theorem \ref{thm:2.1.39}.
\begin{cor}
\label{cor:2.7} Let $U:\mathcal{H}\rightarrow\mathcal{H}$ be an
operator and let $T:=Id+\frac{1+\rho}{2}\left(U-Id\right)$ for some
$\rho\ge0$. Then $U$ is $\rho$-strongly quasi-nonexpansive if and
only if $T$ is a cutter. In particular, $U$ is quasi-nonexpansive
if and only if $T:=\frac{1}{2}\left(U+Id\right)$ is a cutter.
\end{cor}
\begin{defn}
\label{firm=000020def}We say that an operator $T:\mathscr{\mathcal{H}}\rightarrow\mathcal{H}$
is:
\begin{enumerate}
\item \textit{Firmly nonexpansive} if
\[
\left(\forall x,y\in\mathcal{H}\right)\left\langle T\left(x\right)-T\left(y\right),x-y\right\rangle \ge\left\Vert T\left(x\right)-T\left(y\right)\right\Vert ^{2}.
\]
\item $\rho$-\textit{firmly nonexpansive}, where $\rho\ge0$ is a real
number,\textit{ }if 
\[
\left(\forall x,y\in\mathcal{H}\right)\left\Vert T\left(x\right)-T\left(y\right)\right\Vert ^{2}\le\left\Vert x-y\right\Vert ^{2}-\rho\left\Vert \left(x-T\left(x\right)\right)-\left(y-T\left(y\right)\right)\right\Vert ^{2}.
\]
\end{enumerate}
\end{defn}
\begin{rem}
\label{FNE-SQNE}Clearly, for each $\rho\ge0$, a $\rho$-firmly nonexpansive
operator is, in particular, nonexpansive and $\rho$-strongly quasi-nonexpansive.
\end{rem}
\begin{defn}
For $\lambda\in\left[0,2\right]$, an operator $T:\mathcal{H}\rightarrow\mathcal{H}$
is called $\lambda$-\textit{relaxed firmly nonexpansive} if $T$
is a $\lambda$-relaxation of a firmly nonexpansive operator $U$,
that is, $T=U_{\lambda}=\left(1-\lambda\right)Id+\lambda U$.
\end{defn}
\begin{thm}
[{\cite[Theorems 2.2.4 and 2.2.5]{C_book}}]\label{thm:2.2.5} If
$T:\mathscr{\mathcal{H}}\rightarrow\mathcal{H}$ is firmly nonexpansive,
then $T$ is a nonexpansive cutter.
\end{thm}
\begin{thm}
[{\cite[Corollary 2.2.15]{C_book}}]\label{FNE_relax_FNE}For any
$\lambda\in\left(0,2\right]$, an operator $T:\mathcal{H}\rightarrow\mathcal{H}$
is firmly nonexpansive if and only if its relaxation $T_{\lambda}$
is $\left(2-\lambda\right)\lambda^{-1}$-firmly nonexpansive, that
is, if and only if
\[
\left(\forall x,y\in\mathcal{H}\right)\left\Vert T_{\lambda}\left(x\right)-T_{\lambda}\left(y\right)\right\Vert ^{2}\le\left\Vert x-y\right\Vert ^{2}-\left(2-\lambda\right)\lambda^{-1}\left\Vert \left(x-T_{\lambda}\left(x\right)\right)-\left(y-T_{\lambda}\left(y\right)\right)\right\Vert ^{2}.
\]
\end{thm}
The following corollary is immediate from Theorem \ref{FNE_relax_FNE}.
\begin{cor}
\label{cor:=000020frm-nexp}Let $U:\mathcal{H}\rightarrow\mathcal{H}$
be an operator and let $T:=Id+2^{-1}\left(1+\rho\right)\left(U-Id\right)$
for some $\rho\ge0$, that is, $T=U_{2^{-1}\left(1+\rho\right)}$.
Then $U$ is $\rho$-firmly nonexpansive if and only if $T$ is firmly
nonexpansive. In particular, $U$ is nonexpansive if and only if $T:=\frac{1}{2}\left(U+Id\right)$
is firmly nonexpansive.
\end{cor}
\begin{thm}
[{\cite[Theorems 2.2.35 and 2.2.42]{C_book}}]\label{frm-nonex}For
a positive integer $m$, let $\left\{ U_{i}\right\} _{i=1}^{m}$ be
a finite family of $\lambda_{i}$-relaxed firmly nonexpansive operators,
where $U_{i}:\mathcal{H\rightarrow\mathcal{H}}$ and $\lambda_{i}\in\left(0,2\right]$
for each $i=1,2,\dots,m$. Then:
\begin{enumerate}
\item For each finite set of numbers $\left\{ \omega_{i}\right\} _{i=1}^{m}\subset\left[0,1\right]$
such that $\sum_{i=1}^{m}\omega_{i}=1$, the convex combination $U:=\sum_{i=1}^{m}\omega_{i}U_{i}$
is $\sum_{i=1}^{m}\omega_{i}\lambda_{i}$-relaxed firmly nonexpansive.
\item The composition $V:=U_{m}\cdots U_{2}U_{1}$ is $\lambda$-relaxed
firmly nonexpansive, where
\[
\left(0,2\right]\ni\lambda=\begin{cases}
2, & \mathrm{if}\,\max_{i\in\left\{ 1,2,\dots,m\right\} }=2,\\
2\left(\left(\sum_{i=1}^{m}\lambda_{i}\left(2-\lambda_{i}\right)^{-1}\right)^{-1}+1\right)^{-1}, & \mathrm{otherwise}.
\end{cases}
\]
\end{enumerate}
\end{thm}
An immediate consequence of of Theorems \ref{FNE_relax_FNE} and \ref{frm-nonex}
is presented in the following corollary.
\begin{cor}
\label{comp_conv_comb_firm}For a positive integer $m$, let $\left\{ U_{i}\right\} _{i=1}^{m}$
be a finite family of $\rho_{i}$-firmly nonexpansive operators, where
$U_{i}:\mathcal{H\rightarrow\mathcal{H}}$ and $\rho_{i}\in\left[0,\infty\right)$
for each $i=1,2,\dots,m$. Define $\rho:=\min_{i\in\left\{ 1,2,\dots,m\right\} }\rho_{i}$.
Then:
\begin{enumerate}
\item For each finite set of numbers $\left\{ \omega_{i}\right\} _{i=1}^{m}\subset\left[0,1\right]$
such that $\sum_{i=1}^{m}\omega_{i}=1$, the convex combination $U:=\sum_{i=1}^{m}\omega_{i}U_{i}$
is $\rho$- firmly nonexpansive.
\item The composition $V:=U_{m}\cdots U_{2}U_{1}$ is $\rho m^{-1}$- firmly
nonexpansive.
\end{enumerate}
\end{cor}
\begin{example}
\label{metric=000020projection}Given a nonempty, closed and convex
subset $C$ of $\mathcal{H}$, the metric projection $P_{C}$ onto
$C$ is firmly nonexpansive (see, for instance, Theorem 2.2.21 in
\cite{C_book}) and hence, it is a nonexpansive cutter (by Theorem
\ref{thm:2.2.5}). Moreover, $\mathrm{Fix}P_{C}=C$.
\end{example}
\begin{prop}
[{\cite[Propositions 4.5 and 4.6]{CiegZ_app_s}}]\label{Inequalities=000020for=000020compositions=000020and=000020convex=000020combinations}
Let $m$ be a positive integer. For each $i=1,2,\dots,m$ , let $\rho_{i}>0$
and assume that $U_{i}:\mathscr{\mathcal{H}\rightarrow\mathcal{H}}$
is a $\rho_{i}$-strongly quasi nonexpansive operator. Suppose that
$\bigcap_{i=1}^{m}\mathrm{Fix}U_{i}\not=\emptyset$. Let $U:=\sum_{i=1}^{m}\omega_{i}U_{i}$,
where $\omega_{i}\in\left(0,1\right]$ for each $i=1,2,\dots,m$ and
$\sum_{i=1}^{m}\omega_{i}=1$, be a convex combination of $\left\{ U_{i}\right\} _{i=1}^{m}$
and let $V:=U_{m}U_{m-1}\cdots U_{1}$ be a composition of $\left\{ U_{i}\right\} _{i=1}^{m}$.
Moreover, let $x\in\mathscr{\mathcal{H}}$ and let $z\in\bigcap_{i=1}^{m}\mathrm{Fix}U_{i}$
be arbitrary. Then the following inequalities hold:
\[
\frac{1}{2R}\sum_{i=1}^{m}\omega_{i}\rho_{i}\left\Vert U_{i}x-x\right\Vert ^{2}\le\left\Vert Ux-x\right\Vert 
\]
and
\[
\frac{1}{2R}\sum_{i=1}^{m}\rho_{i}\left\Vert S_{i}x-S_{i-1}x\right\Vert ^{2}\le\left\Vert Vx-x\right\Vert ,
\]
for any positive $R\ge\left\Vert x-z\right\Vert $, where $S_{i}=\prod_{j=1}^{i}U_{j}=U_{i}U_{i-1}\cdots U_{1}$
for each $i=1,2,\dots,m$ (by definition, the composition $S_{0}=Id$).
\end{prop}

\subsection{\protect\label{subsec:The-general-bounded}The general bounded regularity
and approximate shrinking properties of operators}

The notions of bounded regularity and approximate shrinking are needed
for establishing the strong convergence of the methods which we discuss
in this paper. The bounded regularity of a finite family of sets was
examined in \cite[Section 5]{BB96} and \cite{BB93}. This property
was extended to an infinite family of sets in \cite{GDSA_inconsist}.
We recall it briefly below.
\begin{defn}
\label{bound_reg}For a nonempty index set $I$ (either finite or
infinite), the family $\left\{ C_{i}\right\} _{i\in I}$ of nonempty,
closed and convex subsets of $\mathcal{H}$ with nonempty intersection
$C$ is \textit{boundedly regular }if for any bounded sequence $\left\{ x^{k}\right\} _{k=0}^{\infty}$
in $\mathcal{H}$, the following implication holds:\textit{
\[
\lim_{k\rightarrow\infty}d\left(x^{k},C_{i}\right)=0\,\,\mathrm{for}\,\,\mathrm{each\,\,}i\in I\,\,\Longrightarrow\,\,\lim_{k\rightarrow\infty}d\left(x^{k},C\right)=0.
\]
}
\end{defn}
The next proposition establishes sufficient conditions for bounded
regularity. Recall that a topological space $X$ is locally compact
if each $x\in X$ has a compact neighborhood with respect to the topology
inherited from $X$.
\begin{prop}
[{\cite[Proposition 3.2]{GDSA_inconsist}}]\label{bounded_reg}Let
$\left\{ C_{i}\right\} _{i\in I}$ be a family of nonempty, closed
and convex subsets of $\mathcal{H}$ with a nonempty intersection
$C$. Then the following assertions hold:
\begin{enumerate}
\item If there is an $i_{0}\in I$ for which the set $C_{i_{0}}$ is a locally
compact topological space (with respect to the norm topology inherited
from $\mathcal{H}$), then the family $\left\{ C_{i}\right\} _{i\in I}$
is boundedly regular.
\item If $\mathcal{H}$ is of finite dimension, then the family $\left\{ C_{i}\right\} _{i\in I}$
is boundedly regular.
\end{enumerate}
\end{prop}
We also use the concept of approximate shrinking, which has been the
subject of extensive study in \cite{CiegZ_app_s}.
\begin{defn}
A quasi-nonexpansive operator $T:\mathcal{H}\rightarrow\mathcal{H}$
\textit{is approximately shrinking} if for each bounded sequence $\left\{ x^{k}\right\} _{k=0}^{\infty}$
in $\mathcal{H}$, the following implication holds:
\[
\lim_{k\rightarrow\infty}\left\Vert T\left(x^{k}\right)-x^{k}\right\Vert =0\,\,\Longrightarrow\,\,\lim_{k\rightarrow\infty}d\left(x^{k},\mathrm{Fix}T\right)=0.
\]
\end{defn}
\begin{example}
\label{ex_as}Given a nonempty, closed and convex subset $C$ of $\mathcal{H}$,
the metric projection $P_{C}$ onto $C$ is approximately shrinking
(see Example 3.5 in \cite{CiegZ_app_s}).
\end{example}

\subsection{\protect\label{subsec:Additional-notions-and}Additional notions
and results which we use in our work}
\begin{defn}
For a nonempty, closed and convex subset $C$ of $\mathcal{H}$, a
sequence $\left\{ x^{k}\right\} _{k=0}^{\infty}$ in $\mathcal{H}$
is
\begin{enumerate}
\item \textit{Fej{\'e}r monotone} w\textit{ith respect to $C$} if for
each $z\in C$ and each $k\in\mathbb{N}$,
\[
\left\Vert x^{k+1}-z\right\Vert \le\left\Vert x^{k}-z\right\Vert .
\]
\item \textit{Strictly Fej{\'e}r monotone} \textit{with respect to $C$}
if for each $z\in C$ and each $k\in\mathbb{N}$,
\[
\left\Vert x^{k+1}-z\right\Vert <\left\Vert x^{k}-z\right\Vert .
\]
\item \textit{Strongly Fej{\'e}r monotone} \textit{with respect to $C$}
if there exists a constant $\alpha>0$ such that 
\[
\left\Vert x^{k+1}-z\right\Vert ^{2}\le\left\Vert x^{k}-z\right\Vert ^{2}-\alpha\left\Vert x^{k+1}-x^{k}\right\Vert ^{2}.
\]
for each $z\in C$ and each $k\in\mathbb{N}$.
\end{enumerate}
\end{defn}
For a discussion of properties of \textit{Fej{\'e}r} monotone sequences
see, for example, Subsection 3.3 in \cite{C_book}.
\begin{thm}
[{\cite[Theorem 2.16(v)]{BB96}}]\label{Fejer}Let $C$ be a nonempty,
closed and convex subset of $\mathcal{H}$. If $\left\{ x^{k}\right\} _{k=0}^{\infty}$
is Fej{\'e}r monotone with respect to $C$, then it converges in
the norm of $\mathcal{H}$ to some point in $C$ if and only if 
\[
\lim_{k\rightarrow\infty}d\left(x^{k},C\right)=0.
\]
\end{thm}
The next definition is fundamental in the analysis of the superiorization
methodology. It defines bounded perturbations resilience of an iterative
algorithm that is governed by an infinite sequence of algorithmic
operators. See \cite{BuRZ1}, where it was first shown that if the
exact iterates of a nonexpansive operator converge, then the inexact
iterates, subject to summable errors, also converge.
\begin{defn}
[{\cite[Definition 2.24]{GDSA_inconsist}}]Let $\Gamma\subseteq\mathcal{H}$
be a given nonempty subset of $\mathcal{H}$ and $\left\{ T_{k}\right\} _{k=0}^{\infty}$
be a sequence of operators, $T_{k}:\mathcal{H}\rightarrow\mathcal{H}$
for each $k\in\mathbb{N}$. The algorithm $x^{k+1}:=T_{k}(x^{k}),$
for all $k\in\mathbb{N},$ is said to be\textit{ bounded perturbations
resilient} with respect to\textbf{ $\Gamma$}\emph{ }if the following
is true: If a sequence $\{x^{k}\}_{k=0}^{\infty},$ generated by the
algorithm, converges in the norm of $\mathcal{H}$ to a point in $\Gamma$
for all $x^{0}\in\mathcal{H}$, then any sequence $\{y^{k}\}_{k=0}^{\infty}$
in $\mathcal{H}$ that is generated by the algorithm $y^{k+1}:=T_{k}(y^{k}+\beta_{k}v^{k}),$
for all $k\in\mathbb{N},$ also converges in the norm of $\mathcal{H}$
to a point in $\Gamma$ for all $y^{0}\in\mathcal{H},$ provided that
$\{\beta_{k}v^{k}\}_{k=0}^{\infty}$ are bounded perturbations, meaning
that $\left\{ \beta_{k}\right\} _{k=0}^{\infty}$ is a sequence of
positive real numbers such that $\sum_{k=0}^{\infty}\beta_{k}<\infty$
and that the vector sequence $\{v^{k}\}_{k=0}^{\infty}$ is a bounded
sequence in $\mathcal{H}$.
\end{defn}
\begin{rem}
In this work we consider inner perturbations, where bounded perturbations
are applied before the evaluation of the underlying operators. An
alternative perturbation model consists of outer perturbations in
which summable error terms are added to the operator outputs; such
perturbations naturally model inexact operator evaluations and have
been discussed in relation to the superiorization method in Section
8.1 in \cite{CensorReem2015}, see Proposition 5 therein. Further
results on the the Krasnosel'ski\u{\i}-Mann iterative method with
perturbations can be found in \cite{Dong2022}. Outer perturbations
in projection methods for the split equality problem appeared in \cite{DongLiHe2018}.
Although inner and outer perturbations are defined differently, under
the nonexpansiveness assumptions imposed throughout our paper they
are asymptotically equivalent in the sense that both lead to the same
convergence behavior of the generated iterates. This shows that the
present analysis is consistent and compatible with results established
in the outer-perturbation framework.
\end{rem}
\begin{thm}
[{\cite[Theorems 3.2 and 5.2]{BuRZ}}]\label{error}Let $C\subset\mathcal{H}$
be a nonempty and closed subset. Let \textup{$\left\{ T_{k}\right\} _{k=0}^{\infty}$},
$T_{k}:\mathscr{\mathcal{H}\rightarrow\mathcal{H}}$ for each $k\in\mathbb{N}$,
be a sequence of nonexpansive operators satisfying $C\subset\cap_{k=0}^{\infty}\mathrm{Fix}T_{k}$.
Assume that for each $y\in\mathcal{H}$ and each $q\in\mathbb{N}$,
the sequence $\left\{ T_{q+k}\cdots T_{q+1}T_{q}\left(y\right)\right\} _{k=0}^{\infty}$
converges in the norm of $\mathcal{H}$ to an element of $C$. Let
$\left\{ \gamma_{k}\right\} _{k=0}^{\infty}\subset\left[0,\infty\right)$
be a real sequence such that $\sum_{\text{k}=0}^{\infty}\gamma_{k}<\infty$
and let $\left\{ y^{k}\right\} _{k=0}^{\infty}\subset\mathcal{H}$.
Further assume that for each $k\in\mathbb{N}$,
\[
\left\Vert y^{k+1}-T_{k}\left(y^{k}\right)\right\Vert \le\gamma_{k}.
\]
Then the sequence $\left\{ y^{k}\right\} _{k=0}^{\infty}$converges
in the norm of $\mathcal{H}$ to an element of $C$.
\end{thm}
The following well-known property of a convex and continuous function
will be used in the sequel.
\begin{thm}
[{\cite[Theorem 16.17(ii)]{BC_book}}]\label{subdiff_ne+b} Let $\phi:\mathcal{H}\rightarrow\mathbb{R}$
be convex and continuous function at the point $x\in\mathcal{H}$.
Then the subgradient set $\partial\phi\left(x\right)$ is nonempty.
\end{thm}

\subsection{\protect\label{GMSA}The general modular string-averaging procedure}

The MSA procedure for a finite number of input operators was introduced
by Reich and Zalas in \cite{MSA}, based on the original string-averaging
scheme introduced by Censor, Elfving and Herman in \cite{CEH2001}.
It was later generalized by Barshad, Gibali and Reich to the GMSA
procedure in \cite{BGR2}, where an infinite number of input operators
along with further properties of this procedure were studied. Below
we recall some of these properties which we need for our results in
the sequel.

Assume that $\left\{ U_{n}\right\} _{n=0}^{\infty}$ is a sequence
of input operators and $\varepsilon\in\left(0,1\right]$. Let $\left\{ N_{k}\right\} _{k=0}^{\infty}$
be a sequence of positive integers which we call ``the sequence of
numbers of recursive steps''. For every $n\in\mathbb{N}$, define
$L_{n}:=\left\{ m\in\mathbb{Z}\mid\,m\le n\right\} $. For each $k\in\mathbb{N}$,
we consider the following settings:
\begin{itemize}
\item Let $c_{k}:\mathbb{N}\backslash\left\{ 0\right\} \cap L_{N_{k}}\rightarrow\left\{ 0,1,2\right\} $
be a function.
\item For each positive integer $n\in L_{N_{k}}$, let $J_{n}^{k}$ be a
nonempty finite subset of $L_{n-1}$ satisfying the following assertion:
if $n\in c_{k}^{-1}\left(\left\{ 0\right\} \right)$, then $\left|J_{n}^{k}\right|=1$
and $J_{n}^{k}\subset L_{0}$.
\item For each $n\in c_{k}^{-1}\left(\left\{ 0\right\} \right)$, let $j_{n}^{k}$
be the unique element of $J_{n}^{k}$ and define $P_{n}^{k}:=1$ .
\item For each $n\in c_{k}^{-1}\left(\left\{ 1\right\} \right)$, let $\omega_{n}^{k}:J_{n}^{k}\rightarrow\left[\varepsilon,1\right]$
be a function which we call ``a weight function'', satisfying $\sum_{j\in J_{n}^{k}}\omega_{n}^{k}\left(j\right)=1$,
and define $P_{n}^{k}:=\left|J_{n}^{k}\right|$.
\item For each $n\in c_{k}^{-1}\left(\left\{ 2\right\} \right)$, let $o_{n}^{k}:\left\{ 1,2,\dots,P_{n}^{k}\right\} \rightarrow J_{n}^{k}$
be a function onto $J_{n}^{k}$, which we call ``an order function'',
and $P_{n}^{k}\ge\left|J_{n}^{k}\right|$ is a positive integer.
\item Let $\alpha_{k}:c_{k}^{-1}\left(\left\{ 0\right\} \right)\rightarrow\left[\varepsilon,2-\varepsilon\right]$
be another function.
\item For each $n\in L_{N_{k}}$, we define a set $I_{n}^{k}\subset\mathbb{N}$
(recursively) by 
\begin{equation}
I_{n}^{k}:=\begin{cases}
\begin{split}\left\{ -n\right\} \end{split}
 & \mathrm{\,\,if}\,\,n\le0,\\
\cup_{j\in J_{n}^{k}}I_{j}^{k} & \mathrm{\,\,if}\,\,n>0.
\end{cases}\label{eq:}
\end{equation}
\end{itemize}
Now we are ready to describe our GMSA procedure.

For each $k\in\mathbb{N}$, define a sequence $\left\{ V_{n}^{k}\right\} _{n=-\infty}^{N_{k}}$
of intermediate modules $V_{n}^{k}$ (recursively) with respect to
$N_{k}$ recursive steps by 
\begin{equation}
V_{n}^{k}:=\begin{cases}
V_{n}^{k}:=U_{-n} & \mathrm{\,\,if}\,\,n\le0,\\
(\mathrm{Relaxation)\,\,}Id+\alpha_{k}\left(n\right)\left(V_{j_{n}^{k}}^{k}-Id\right), & \,\mathrm{\,if}\,\,n\in c_{k}^{-1}\left(\left\{ 0\right\} \right),\\
(\mathrm{Convex\,\,combination})\,\,\sum_{j\in J_{n}^{k}}\omega_{n}^{k}\left(j\right)V_{j}^{k}, & \,\,\mathrm{if}\,\,n\in c_{k}^{-1}\left(\left\{ 1\right\} \right),\\
\mathrm{(Composition)\,\,}V_{o_{n}^{k}\left(P_{n}^{k}\right)}^{k}\cdots V_{o_{n}^{k}\left(2\right)}^{k}V_{o_{n}^{k}\left(1\right)}^{k}, & \,\,\mathrm{if}\,\,n\in c_{k}^{-1}\left(\left\{ 2\right\} \right)
\end{cases}\label{eq:-3}
\end{equation}
for each $n\in L_{N_{k}}$. Following the work described in \cite{MSA},
we say that the sequence $\left\{ T_{k}\right\} _{k=0}^{\infty}$,
defined by $T_{k}:=V_{N_{k}}^{k}$ for each $k\in\mathbb{N}$, is
``generated by the GMSA procedure''.

Note that for each $k\in\mathbb{N}$ and each positive integer $n\in L_{N_{k}}$,
the set $J_{n}^{k}$ stores the indices of all the previous modules
which were used in order to define the module $V_{n}^{k}$, while
for each $n\in L_{N_{k}}$, the set $I_{n}^{k}$ stores the indices
of those input operators which were actually applied in the construction
of $V_{n}^{k}$.
\begin{lem}
[{\cite[Lemma 3.1]{BGR2}}]\label{lem1}Assume that $\left\{ \ell_{k}\right\} _{k=0}^{\infty}$
is a strictly increasing sequence of natural numbers. For each $k\in\mathbb{N}$
and $n\in L_{N_{\ell_{k}}}$, Define $N_{k}^{\prime}:=N_{\ell_{k}}$,
$J_{n}^{\prime k}:=J_{n}^{\ell_{k}}$, $c_{k}^{\prime}:=c_{\ell_{k}},$
$\alpha_{k}^{\prime}:=\alpha_{\ell_{k}}$, $P_{n}^{\prime k}:=P_{n}^{\ell_{k}}$,
$j_{n}^{\prime k}:=j_{n}^{\ell_{k}}$ the unique element of $J_{n}^{\prime k}$
if $n\in c_{k}^{\prime-1}\left(\left\{ 0\right\} \right)$, $\omega_{n}^{\prime k}:=\omega_{n}^{\ell_{k}}$
if $n\in c_{k}^{\prime-1}\left(\left\{ 1\right\} \right)$ and $o_{n}^{\prime k}:=o_{n}^{\ell_{k}}$
if $n\in c_{k}^{\prime-1}\left(\left\{ 2\right\} \right)$. Then for
each $k\in\mathbb{N}$ and each $n\in L_{N_{k}^{\prime}}$, we have
$I_{n}^{\prime k}=I_{n}^{\ell_{k}}$ and $V_{n}^{\prime k}=V_{n}^{\ell_{k}}$,
where the family $\left\{ \left\{ I_{n}^{\prime k}\right\} _{n\in L_{N_{k}^{\prime}}}\right\} _{k=0}^{\infty}$
is defined by \eqref{eq:} with respect to the sets $\left\{ J_{n}^{\prime k}\right\} _{k\in\mathbb{N},\,n\in L_{N_{k}^{\prime}}}$and
the family $\left\{ \left\{ V_{n}^{\prime k}\right\} _{n=-\infty}^{N_{k}}\right\} _{k\in\mathbb{N}}$
is defined by \eqref{eq:-3} with respect to the sequence $\left\{ N_{k}^{\prime}\right\} _{k=0}^{\infty}$
and the above parameters. Consequently, the sequence $\left\{ S_{k}\right\} _{k=0}^{\infty}$
that is generated by the GMSA procedure with respect to the sequence
$\left\{ N_{k}^{\prime}\right\} _{k=0}^{\infty}$ and the above parameters
satisfies $S_{k}=T_{\ell_{k}}$ for each $k\in\mathbb{N}$.
\end{lem}
\begin{lem}
[{\cite[Lemma 3.2]{BGR2}}]\label{lem2} In the settings of the GMSA
procedure above, for each $k\in\mathbb{N}$ and each non-positive
integer $n$ such that $U_{-n}$ is a $2^{-1}$-strongly quasi-nonexpansive
operator, the intermediate module $V_{n}^{k}$, generated by the GMSA
procedure, is $2^{-1}\varepsilon$-strongly quasi-nonexpansive and
\[
\mathrm{Fix}V_{n}^{k}=\cap_{i\in I_{n}^{k}}\mathrm{Fix}U_{i}.
\]
\end{lem}
\begin{lem}
[{\cite[Lemma 3.4]{BGR2}}]\label{lem3}Let $\left\{ U_{n}\right\} _{n=0}^{\infty}$
be a sequence of $2^{-1}$-strongly quasi-nonexpansive operators such
that $\cap_{i\in I_{n}^{k}}\mathrm{Fix}U_{i}\not=\emptyset$ for each
$n\in\mathbb{N}$, let $k\in\mathbb{N}$ and let $n\in L_{N_{k}}$.
Then 
\[
\mathrm{Fix}V_{n}^{k}=\cap_{i\in I_{n}^{k}}\mathrm{Fix}U_{i}\not=\emptyset.
\]
Moreover, if $n$ is positive, then we have
\[
\cap_{j\in J_{n}^{k}}\mathrm{Fix}V_{j}^{k}=\cap_{i\in I_{n}^{k}}\mathrm{Fix}U_{i}\not=\emptyset
\]
and the following assertions hold:
\begin{enumerate}
\item If $n\in c_{k}^{-1}\left(\left\{ 0\right\} \right)$ and, in addition,
$U_{-j_{n}^{k}}$ is a cutter, or if $\alpha_{k}\left(n\right)=1$,
then the intermediate module $V_{n}^{k},$ generated by the GMSA procedure,
is an $\frac{\varepsilon}{2\Pi_{i=1}^{n}P_{i}^{k}}$-strongly quasi-nonexpansive
operator.
\item If $n\in c_{k}^{-1}\left(\left\{ 1\right\} \right)\cup c_{k}^{-1}\left(\left\{ 2\right\} \right)$,
then the intermediate module $V_{n}^{k}$, generated by the GMSA procedure,
is an $\frac{\varepsilon}{2\Pi_{i=1}^{n}P_{i}^{k}}$-strongly quasi-nonexpansive
operator.
\end{enumerate}
We use the convention that for a non-positive integer $n$, the empty
product $\Pi_{i=1}^{n}P_{i}^{k}=1$.
\end{lem}
Under the assumptions made in Subsection \ref{GMSA}, we consider
the following condition on the sequence $\left\{ I_{N_{K}}^{k}\right\} _{k=0}^{\infty}$,
which is known in the literature as an ``admissible control'' (see,
for example, \cite{PLC}):

For each $n\in\mathbb{N}$, there is an integer $M_{n}>0$ such that
\begin{equation}
n\in\bigcup_{k=i}^{i+M_{n}-1}I_{N_{k}}^{k}\,\,\mathrm{for\,all}\,i\in\mathbb{N}.\label{eq:-16}
\end{equation}

In the case where for each $n\in\mathbb{N}$, $M_{n}$ in\textbf{
}the condition of\textbf{ $\eqref{eq:-16}$} equals some $s>0$, this
condition is known in the literature as an ``s-intermittent control''
(see, for instance, \cite[Definition 5.8.10]{C_book}, \cite{BRZ(R)}
and \cite{MSA}).

The next example (see Example 4.3 in \cite{BGR2}) demonstrates a
particular structure of the GMSA procedure, where \eqref{eq:-16}
is satisfied. It should be noted that the GMSA studied here is a general
modular algorithmic scheme which can cover algorithmic structures
investigated earlier. In this respect, one can find examples of similar
nature elsewhere, for instance, in \cite{Combettes1997}.
\begin{example}
\label{Concrete=000020example} For each $i\in\mathbb{N}$, let $\mathrm{max}_{i}$
be a maximal natural number such that $i+1=2^{\mathrm{max}_{i}}p$,
where $p\in\mathbb{N}$ is odd. Define a sequence $\left\{ f_{i}\right\} _{i=0}^{\infty}$
of natural numbers by $f_{i}:=\mathrm{max}_{i}$. For each $n\in\mathbb{N}$,
Define $M_{n}:=2^{n+1}$. Then for each $n,i\in\mathbb{N}$, we have
$n\in f\left(\left\{ i,i+1,\dots,i+M_{n}-1\right\} \right)$ because
$n=f_{2^{n}\left(2m+1\right)-1}$ for each $m\in\mathbb{N}$. The
first $20$ elements of $\left\{ f_{i}\right\} _{i=0}^{\infty}$ are:
\[
0,1,0,2,0,1,0,3,0,1,0,2,0,1,0,4,0,1,0,2.
\]
Now for each $k\in\mathbb{N}$, let $c_{k}:\left\{ 1\right\} \rightarrow\left\{ 0,1,2\right\} $
be a function. Define a sequence $\left\{ N_{k}\right\} _{k=0}^{\infty}$
by $N_{k}:=1$ for each $k\in\mathbb{N}$. Clearly, $L_{N_{k}}\cap\mathbb{N}\backslash\left\{ 0\right\} =\left\{ 1\right\} $.
Pick a bounded sequence $\left\{ P_{1}^{k}\right\} _{k=0}^{\infty}$
of positive integers and a sequence $\left\{ J_{1}^{k}\right\} _{k=0}^{\infty}$
of subsets of $L_{0}$ such that $-f_{k}\in J_{1}^{k}$, $\left|J_{1}^{k}\right|\le P_{1}^{k}$
and $\left|J_{1}^{k}\right|=\begin{cases}
1, & \mathrm{if}\,c_{k}\left(1\right)=0,\\
P_{1}^{k}, & \mathrm{if}\,c_{k}\left(1\right)=1,
\end{cases}$ for each $k\in\mathbb{N}$. Define the rest of parameters of GMSA
procedure presented in Section \ref{GMSA}. Clearly, then $I_{N_{k}}^{k}=-\text{\ensuremath{J_{1}^{k}}}$,
where $-J_{1}^{k}=\left\{ -j|\,j\in J_{1}^{k}\right\} $, and hence
for each $k\in\mathbb{N}$ and each $n\in\mathbb{N}$, Equation \eqref{eq:-16}
is satisfied.
\end{example}
In the next example we recall the MSA procedure presented in \cite{MSA},
where the finite sequence of input operators $\left\{ U_{i}^{\prime}\right\} _{i=0}^{\mathcal{M}}$
was considered for some positive integer $\mathcal{M}$ and $U_{0}^{\prime}$
was set to be the identity $Id$. We show that it is a particular
case of the GMSA procedure above. See also Example 4.4 in \cite{BGR2}
in this connection.
\begin{example}
[MSA methods with an admissible control] \label{MSA} For a finite
sequence of operators $\left\{ U_{i}^{\prime}\right\} _{i=0}^{\mathcal{M}}$
(where $\mathcal{M}\ge0$ is an integer and $U_{i}:\mathcal{H\rightarrow\mathcal{H}}$
for each $i\in\left\{ 1,2,\dots,\mathcal{M}\right\} $) and a sequence
$\left\{ N_{k}^{\prime}\right\} _{k=0}^{\infty}$ of positive integers,
the parameters of the MSA procedure are defined similarly, where the
sets $\left\{ \left\{ J_{n}^{\prime k}\right\} _{n\in\mathbb{N}\backslash\left\{ 0\right\} \cap L_{N_{k}^{\prime}}}\right\} _{k\in\mathbb{N}}$
are the subsets of $\left\{ -\mathcal{M},-\mathcal{M}+1,\dots,n-1\right\} $
and the intermediate modules $\left\{ \left\{ V_{n}^{\prime k}\right\} _{n=-\mathcal{M}}^{N_{k}}\right\} _{k\in\mathbb{N}}$
are considered (instead of $\left\{ \left\{ V_{n}^{\prime k}\right\} _{n=-\infty}^{N_{k}}\right\} _{k\in\mathbb{N}}$)
and, consequently, the finite sequence of sets of indices $\left\{ \left\{ I_{n}^{\prime N_{k}^{\prime}}\right\} _{n=-\mathcal{M}}^{N_{k}^{\prime}}\right\} _{k\in\mathbb{N}}$
are of interest (instead of $\left\{ \left\{ I_{n}^{\prime N_{k}^{\prime}}\right\} _{n=-\infty}^{N_{k}^{\prime}}\right\} _{k\in\mathbb{N}}$).
Let $\left\{ f_{n}\right\} _{n=0}^{\infty}$ be a sequence defined
in Example \ref{Concrete=000020example}. Define a sequence $\left\{ U_{n}\right\} _{n=0}^{\infty}$
by
\[
U_{n}:=\begin{cases}
U_{n}^{\prime}, & \mathrm{if}\,n\in\left\{ 0,1,\dots,\mathcal{M}\right\} ,\\
Id, & \mathrm{if}\,n\not\in\left\{ 0,1,\dots,\mathcal{M}\right\} ,
\end{cases}
\]
for each $n\in\mathbb{N}$, and define a sequence $\left\{ N_{k}\right\} _{k=0}^{\infty}$
by
\[
N_{k}:=\begin{cases}
N_{k}^{\prime}, & \mathrm{if}\,k\in f^{-1}\left(\left\{ 0,1,\dots,\mathcal{M}\right\} \right),\\
N_{k}^{\prime}+1, & \mathrm{if}\,k\not\in f^{-1}\left(\left\{ 0,1,\dots,\mathcal{M}\right\} \right),
\end{cases}
\]
for each $k\in\mathbb{N}$. Define

\[
J_{n}^{k}:=\begin{cases}
J_{n}^{\prime k}, & \mathrm{if}\,n\in L_{N_{k}^{\prime}},\\
\left\{ N_{k}^{\prime},-f_{k}\right\} , & \mathrm{if}\,n=N_{k}^{\prime}+1,
\end{cases}
\]
and 
\[
P_{n}^{k}:=\begin{cases}
P_{n}^{\prime k}, & \mathrm{if}\,n\in L_{N_{k}^{\prime}},\\
2, & \mathrm{if}\,n=N_{k}^{\prime}+1,
\end{cases}
\]
for each $k\in\mathbb{N}$ and each positive integer $n\in L_{N_{k}}$.
For each $k\in\mathbb{N}$, define a function $c_{k}:\mathbb{N}\backslash\left\{ 0\right\} \cap L_{N_{k}}\rightarrow\left\{ 0,1,2\right\} $
by 
\[
c_{k}\left(n\right):=\begin{cases}
c_{k}^{\prime}\left(n\right), & \mathrm{if}\,n\in L_{N_{k}^{\prime}},\\
2, & \mathrm{if}\,n=N_{k}^{\prime}+1,
\end{cases}
\]
for each positive integer $n\in L_{N_{k}}$, and for each $k\not\in f^{-1}\left(\left\{ 0,1,\dots,\mathcal{M}\right\} \right)$,
define a function $o_{N_{k}}^{k}:\left\{ 1,2\right\} \rightarrow J_{N_{k}}^{k}$
by 
\[
o_{N_{k}}^{k}\left(m\right):=\begin{cases}
N_{k}^{\prime}, & \mathrm{if}\,m=1,\\
-f_{k}, & \mathrm{if}\,m=2,
\end{cases}
\]
for each $m\in\left\{ 1,2\right\} $. Now set the rest of parameters
of the GMSA procedure with respect to the operators $\left\{ U_{n}\right\} _{n=0}^{\infty}$
to be the same as the parameters of the MSA procedure above. By \eqref{eq:}.
\begin{equation}
I_{N_{k}}^{k}=\begin{cases}
I_{N_{k}^{\prime}}^{\prime k}, & \mathrm{if}\,k\in f^{-1}\left(\left\{ 0,1,\dots,\mathcal{M}\right\} \right),\\
I_{N_{k}^{\prime}}^{\prime k}\cup\left\{ f_{k}\right\} , & \mathrm{if}\,k\not\in f^{-1}\left(\left\{ 0,1,\dots,\mathcal{M}\right\} \right).
\end{cases}\label{eq:-9}
\end{equation}
By \eqref{eq:-3}, the sequence $\left\{ T_{k}^{\prime}\right\} _{k=0}^{\infty}$
of the operators generated by the MSA procedure above satisfies $T_{k}^{\prime}=V_{N_{k}^{\prime}}^{\prime k}=V_{N_{k}}^{k}=T_{k}$
for each $k\in\mathbb{N}$. If, in addition, the sequence $\left\{ I_{N_{k}}^{\prime k}\right\} _{k=0}^{\infty}$
satisfies Equation\textbf{ }\eqref{eq:-16} with respect to the set
$\left\{ 0,1,\dots,\mathcal{M}\right\} $, that is, for each $n\in\left\{ 0,1,\dots,\mathcal{M}\right\} $,
Equation \eqref{eq:-16} holds, then it is easily verified by \eqref{eq:-9}
that the sequence $\left\{ I_{N_{k}}^{k}\right\} _{k=0}^{\infty}$
satisfies Equation \eqref{eq:-16} with respect to $\mathbb{N}$.
\end{example}
\begin{rem}
\label{fix_cl_cnv}Note that if for each $n\in\mathbb{N}$, the operator
$U_{n}$ has a fixed point and is $\rho$-strongly nonexpansive for
some $\rho\ge0$, then, by Theorem \ref{thm:2.1.39}, Remark \ref{Relaxation=000020has=000020the=000020same=000020Fix}
and Remark \ref{cutter_fix}, $\mathrm{Fix}U_{n}$ is closed and convex
for each $n\in\mathbb{N}$ and, consequently, $\cap_{n\in\mathbb{N}}\mathrm{Fix}U_{n}$
is closed and convex.
\end{rem}
Throughout the rest of the paper we consider the settings of the GMSA
procedure presented in this section.

\section{\protect\label{sec:The-strong-convergence}The strong convergence
properties of the GMSA procedure based methods}

In this section we investigate the strong convergence and the bounded
perturbation resilience of Algorithm \ref{alg=0000201} with respect
to the operators $\left\{ T_{k}\right\} _{k=0}^{\infty}$ generated
by the GMSA procedure described in Section \ref{GMSA}, and deduce
important corollaries for the case where the input operators are certain
relaxations of firmly nonexpansive ones.

\subsection{\protect\label{str_quas_non}General results concerning strongly
quasi-nonexpansive input operators}

We begin by proving the following key auxiliary lemma.
\begin{lem}
\label{lem5}Let $\varepsilon\in\left(0,1\right]$, let $\left\{ U_{n}\right\} _{n=0}^{\infty}$
be a sequence of $2^{-1}$-strongly quasi-nonexpansive and approximately
shrinking operators and suppose that the sequence of positive integers
$\left\{ N_{k}\right\} _{k=0}^{\infty}$ is bounded. Let $\left\{ T_{k}\right\} _{k=0}^{\infty}$
be a sequence of operators generated by the GMSA procedure as in \eqref{eq:-3}.
Assume also that for each $k\in\mathbb{N}$ and each $n\in c_{k}^{-1}\left(\left\{ 0\right\} \right)$,
the operator $U_{j_{n}^{k}}$ is a cutter, or that $\alpha_{k}\left(n\right)=1$.
Further assume that there exists $r>0$ such that for each $k\in\mathbb{N}$
and each $n\in c_{k}^{-1}\left(\left\{ 1\right\} \right)\cup c_{k}^{-1}\left(\left\{ 2\right\} \right)$,
we have $\cap_{i\in I_{n}^{k}}\mathrm{Fix}U_{i}\cap B\left(0,r\right)\not=\emptyset$,
that is, the finitely many members of the sequence $\left\{ U_{n}\right\} _{n=0}^{\infty}$
with a common fixed point in the ball $B\left(0,r\right)$ are used
in the construction of the intermediate modules $V_{n}^{k}$. Then
for any bounded sequence $\left\{ x^{k}\right\} _{k=0}^{\infty}\subset\mathcal{H}$,
the following implication holds:
\[
T_{k}\left(x^{k}\right)-x^{k}\rightarrow0\,\,\Longrightarrow\lim_{k\rightarrow\infty}\max_{i\in I_{N_{k}}^{k}}d\left(x^{k},\mathrm{Fix}U_{i}\right)=0.
\]
\end{lem}
\begin{proof}
Since the sequence $\left\{ N_{k}\right\} _{k=0}^{\infty}$ is a bounded
sequence of positive integers, it attains only a finite number of
values. As a result, passing to subsequence if necessary and using
Lemma \ref{lem1}, we may assume without any loss of generality that
$N_{k}=N$ for all $k\in\mathbb{N}$, where $N$ is some positive
integer. Clearly, it is enough to show the following assertion:

For each positive integer $n\in L_{N}$, we have the implication (which
holds, in particular, for $n=N$)
\begin{equation}
V_{n}^{k}\left(x^{k}\right)-x^{k}\rightarrow0\,\,\Longrightarrow\,\,\lim_{k\rightarrow\infty}\max_{i\in I_{n}^{k}}d\left(x^{k},\mathrm{Fix}U_{i}\right)=0\label{eq:-6}
\end{equation}
for each bounded sequence $\left\{ x^{k}\right\} _{k=0}^{\infty}\subset\mathcal{H}$.
Let $n$ be a positive integer in $L_{N}$. We prove that $n$ satisfies
implication \eqref{eq:-6} by induction on the set $\mathbb{N}\backslash\left\{ 0\right\} \cap L_{N}$.
Assume that for each positive integer $j<n$ in $L_{N}$, this implication
holds for each bounded sequence $\left\{ x^{k}\right\} _{k=0}^{\infty}\subset\mathcal{H}$.
Let $\left\{ x^{k}\right\} _{k=0}^{\infty}\subset\mathcal{H}$ be
bounded. Suppose that 
\begin{equation}
V_{n}^{k}\left(x^{k}\right)-x^{k}\rightarrow0.\label{eq:-11}
\end{equation}
By the definition of the GMSA procedure and since for each $k\in\mathbb{N}$,
the image of the function $c_{k}$ is contained in the finite set
$\left\{ 0,1,2\right\} $, we consider, passing to subsequence if
necessary and by Lemma \ref{lem1}, the following three cases:

\textbf{\textit{Case 1:}} $c_{k}\left(n\right)=0$ for each $k\in\mathbb{N}$.
In this case (by \eqref{eq:-3}) we have for each $k\in\mathbb{N}$,
\begin{equation}
V_{n}^{k}=Id+\alpha_{k}\left(n\right)\left(V_{j_{n}^{k}}^{k}-Id\right)=Id+\alpha_{k}\left(n\right)\left(U_{-j_{n}^{k}}-Id\right),\label{eq:-7}
\end{equation}
where $\alpha_{k}\left(n\right)\in\left[\varepsilon,2-\varepsilon\right]$
and $j_{n}^{k}\in J_{n}^{k}\subset L_{0}$. Clearly, $J_{n}^{k}=\left\{ j_{n}^{k}\right\} $.
By \eqref{eq:}, 
\begin{equation}
I_{n}^{k}=\cup_{j\in J_{n}^{k}}I_{j}^{k}=I_{j_{n}^{k}}^{k}=\left\{ -j_{n}^{k}\right\} \label{eq:-8}
\end{equation}
for each $k\in\mathbb{N}$. By \eqref{eq:-7}, for each $k\in\mathbb{N}$,
we have
\[
\left\Vert V_{n}^{k}\left(x^{k}\right)-x^{k}\right\Vert =\alpha_{k}\left(n\right)\left\Vert U_{-j_{n}^{k}}\left(x^{k}\right)-x^{k}\right\Vert \ge\varepsilon\left\Vert U_{-j_{n}^{k}}\left(x^{k}\right)-x^{k}\right\Vert 
\]
and, hence, by \eqref{eq:-11}, $\left\Vert U_{-j_{n}^{k}}\left(x^{k}\right)-x^{k}\right\Vert \rightarrow0$.
Since $U_{-j_{n}^{k}}$ is approximately shrinking, it follows from
\eqref{eq:-8} that 
\[
\lim_{k\rightarrow\infty}\max_{i\in I_{n}^{k}}d\left(x^{k},\mathrm{Fix}U_{i}\right)=\lim_{k\rightarrow\infty}d\left(x^{k},\mathrm{Fix}U_{-j_{n}^{k}}\right)=0,
\]
which proves \eqref{eq:-6}.

For the next two cases, assume that $c_{k}\left(n\right)\not=0$ for
each $k\in\mathbb{N}$ and let $\left\{ z^{k}\right\} _{k=0}^{\infty}$
be a sequence such that $z^{k}\in\cap_{i\in I_{n}^{k}}\mathrm{Fix}U_{i}\cap B\left(0,r\right)$
for each $k\in\mathbb{N}$ and $R>0$ so that $\left\Vert x^{k}-z^{k}\right\Vert <R$
for all $k\in\mathbb{N}$ is satisfied.

\textbf{\textit{Case 2:}}\textit{ }$c_{k}\left(n\right)=1$ for each
$k\in\mathbb{N}$. In this case (by \eqref{eq:-3}) $V_{n}^{k}=\sum_{j\in J_{n}^{k}}\omega_{n}^{k}\left(j\right)V_{j}^{k}$,
where $J_{n}^{k}\subset L_{n-1}$ and the weights $\left\{ \omega_{n}^{k}\left(j\right)\right\} _{j\in J_{n}^{k}}$
are in the interval $\left[\varepsilon,1\right]$ for each $k\in\mathbb{N}$.
Moreover, $I_{n}^{k}=\cup_{j\in J_{n}^{k}}I_{j}^{k}$ for each $k\in\mathbb{N}$
(by \eqref{eq:}). By Lemmata \ref{lem2} and \ref{lem3}\textit{(ii)},
$V_{j}^{k}$ is $\frac{\varepsilon}{2\Pi_{i=1}^{j}P_{i}^{k}}$-strongly-quasi
nonexpansive and hence $\frac{\varepsilon}{2\Pi_{i=1}^{n}P_{i}^{k}}$-strongly-quasi
nonexpansive, for each $k\in\mathbb{N}$ and each $j\in J_{n}^{k}$,
and by Lemma \ref{lem3},
\[
\cap_{j\in J_{n}^{k}}\mathrm{Fix}V_{j}^{k}=\cap_{i\in I_{n}^{k}}\mathrm{Fix}U_{i}\not=\emptyset
\]
 for each $k\in\mathbb{N}$. By Proposition \ref{Inequalities=000020for=000020compositions=000020and=000020convex=000020combinations}
and \eqref{eq:-11}, for each $k\in\mathbb{N}$, we have
\[
\left\Vert V_{n}^{k}\left(x^{k}\right)-x^{k}\right\Vert \ge\frac{\varepsilon^{2}}{4R\Pi_{i=0}^{n}P_{i}^{k}}\sum_{j\in J_{n}^{k}}\left\Vert V_{j}^{k}\left(x^{k}\right)-x^{k}\right\Vert ^{2}\rightarrow0,
\]
which implies $\left\Vert V_{j}^{k}\left(x^{k}\right)-x^{k}\right\Vert ^{2}\rightarrow0$
for each $j\in J_{n}^{k}$. Now, by the induction hypothesis and \eqref{eq:}
(since $U_{-j}$ is approximately shrinking for a non-positive $j$),
$\lim_{k\rightarrow\infty}\max_{i\in I_{j}^{k}}d\left(x^{k},\mathrm{Fix}U_{i}\right)=0$
for each $j\in J_{n}^{k}$. It follows then from \eqref{eq:} that
\[
\lim_{k\rightarrow\infty}\max_{i\in I_{n}^{k}}d\left(x^{k},\mathrm{Fix}U_{i}\right)=0.
\]

\textbf{\textit{Case 3:}}\textit{ }$c_{k}\left(n\right)=2$ for each
$k\in\mathbb{N}$. In this case (by \eqref{eq:-3}) for each $k\in\mathbb{N}$,
we have $V_{n}^{k}=V_{o_{n}^{k}\left(P_{n}^{k}\right)}^{k}\cdots V_{o_{n}^{k}\left(2\right)}^{k}V_{o_{n}^{k}\left(1\right)}^{k}$,
where $o_{n}^{k}:\left\{ 1,2,\dots,P_{n}^{k}\right\} \rightarrow J_{n}^{k}$
is an order function and $J_{n}^{k}\subset L_{n-1}$. Moreover, $I_{n}^{k}=\cup_{j\in J_{n}^{k}}I_{j}^{k}$
for each $k\in\mathbb{N}$. By Lemmata \ref{lem2} and \ref{lem3}\textit{(ii)},
for each $k\in\mathbb{N}$ and each $q=1,2,\dots,P_{n}^{k}$, the
operator $V_{o_{n}^{k}\left(q\right)}^{k}$ is $\frac{\varepsilon}{2\Pi_{i=1}^{o_{n}^{k}\left(q\right)}P_{i}^{k}}$-strongly-quasi
nonexpansive and hence $\frac{\varepsilon}{2\Pi_{i=1}^{n}P_{i}^{k}}$-strongly-quasi
nonexpansive. By Lemma \ref{lem3},
\begin{equation}
\cap_{q=1}^{P_{n}^{k}}\mathrm{Fix}V_{o_{n}^{k}\left(q\right)}^{k}=\cap_{j\in J_{n}^{k}}\mathrm{Fix}V_{j}^{k}=\cap_{i\in I_{n}^{k}}\mathrm{Fix}U_{i}\not=\emptyset.\label{eq:-13}
\end{equation}
 By Proposition \ref{Inequalities=000020for=000020compositions=000020and=000020convex=000020combinations}
and \eqref{eq:-11}, for each $k\in\mathbb{N}$, we have
\begin{equation}
\left\Vert V_{n}^{k}\left(x^{k}\right)-x^{k}\right\Vert \ge\frac{\varepsilon}{4R\Pi_{i=0}^{n}P_{i}^{k}}\sum_{q=1}^{P_{n}^{k}}\left\Vert S_{q}\left(x^{k}\right)-S_{q-1}\left(x^{k}\right)\right\Vert \rightarrow0,\label{eq:-10}
\end{equation}
where $S_{q}=\prod_{i=1}^{q}V_{o_{n}^{k}\left(q\right)}^{k}=V_{o_{n}^{k}\left(q\right)}^{k}\cdots V_{o_{n}^{k}\left(2\right)}V_{o_{n}^{k}\left(1\right)}$
for each $q=1,2,\dots,P_{n}^{k}$. Let $j\in J_{n}^{k}$ be arbitrary.
Since the mapping $o_{n}^{k}$ is onto $J_{n}^{k}$, there is $q_{0}\in\left\{ 1,2,\dots,P_{n}^{k}\right\} $
such that $j=o_{n}^{k}\left(q_{0}\right)$. This, combined with \eqref{eq:-10},
implies that for each $k\in\mathbb{N}$,
\begin{align}
\left\Vert S_{q_{0}}\left(x^{k}\right)-S_{q_{0}-1}\left(x^{k}\right)\right\Vert  & =\left\Vert V_{o_{n}^{k}\left(q_{0}\right)}^{k}S_{q_{0}-1}\left(x^{k}\right)-S_{q_{0}-1}\left(x^{k}\right)\right\Vert \nonumber \\
 & =\left\Vert V_{j}^{k}S_{q_{0}-1}\left(x^{k}\right)-S_{q_{0}-1}\left(x^{k}\right)\right\Vert \rightarrow0.\label{eq:-2}
\end{align}
Since for each $k\in\mathbb{N}$, the operator $S_{q_{0}}$ is quasi-nonexpansive
(by \eqref{eq:-13}, Lemma \ref{lem2}, Lemma \ref{lem3}\textit{(ii)}
and Theorem \ref{2.1.50}\textit{(ii)}), we see, from \eqref{eq:-13},
that
\[
\left\Vert S_{q_{0}-1}\left(x^{k}\right)-z^{k}\right\Vert \le\left\Vert x^{k}-z^{k}\right\Vert <R
\]
for each $k\in\mathbb{N}$, which implies that the sequence $\left\{ S_{q_{0}-1}\left(x^{k}\right)\right\} _{k=0}^{\infty}$
is bounded because the sequence $\left\{ z^{k}\right\} _{k=0}^{\infty}$
is bounded. By \eqref{eq:-2}, the induction hypothesis and \eqref{eq:}
(since $U_{-j}$ is approximately shrinking if $j$ is non-positive),
we obtain 
\begin{equation}
\lim_{k\rightarrow\infty}\max_{i\in I_{j}^{k}}d\left(S_{q_{0}-1}\left(x^{k}\right),\mathrm{Fix}U_{i}\right)=0.\label{eq:-12}
\end{equation}
 Now for each $k\in\mathbb{N}$, since $\cap_{i\in I_{n}^{k}}\mathrm{Fix}U_{i}\not=\emptyset$,
by Remark \ref{fix_cl_cnv}, the properties of the metric projection,
the triangle inequality, \eqref{eq:-10} and \eqref{eq:-12}, we have
\begin{align*}
\max_{i\in I_{j}^{k}}d\left(x^{k},\mathrm{Fix}U_{i}\right) & =\max_{i\in I_{j}^{k}}\left\Vert x^{k}-P_{\mathrm{Fix}U_{i}}\left(x^{k}\right)\right\Vert \le\max_{i\in I_{j}^{k}}\left\Vert x^{k}-P_{\mathrm{Fix}U_{i}}S_{q_{0}-1}\left(x^{k}\right)\right\Vert \\
 & \le\max_{i\in I_{j}^{k}}\left\Vert S_{q_{0}-1}\left(x^{k}\right)-P_{\mathrm{Fix}U_{i}}S_{q_{0}-1}\left(x^{k}\right)\right\Vert +\left\Vert S_{q_{0}-1}\left(x^{k}\right)-x^{k}\right\Vert \\
 & \le\max_{i\in I_{j}^{k}}\left\Vert S_{q_{0}-1}\left(x^{k}\right)-P_{\mathrm{Fix}U_{i}}S_{q_{0}-1}\left(x^{k}\right)\right\Vert +\sum_{q=1}^{q_{0}-1}\left\Vert S_{q}\left(x^{k}\right)-S_{q-1}\left(x^{k}\right)\right\Vert \\
 & =\max_{i\in I_{j}^{k}}d\left(S_{q_{0}-1}\left(x^{k}\right),\mathrm{Fix}U_{i}\right)+\sum_{q=1}^{q_{0}-1}\left\Vert S_{q}\left(x^{k}\right)-S_{q-1}\left(x^{k}\right)\right\Vert \rightarrow0.
\end{align*}
As a result, by \eqref{eq:} , $\lim_{k\rightarrow\infty}\max_{i\in I_{n}^{k}}d\left(x^{k},\mathrm{Fix}U_{i}\right)=0$.
Lemma \ref{lem5} is now proved.
\end{proof}
The next theorem specifies conditions under which the sequence $\left\{ x^{k}\right\} _{k=0}^{\infty}$,
generated by Algorithm \ref{alg=0000201}, is Fej{\'e}r monotone,
converges in the norm of $\mathcal{H}$ to a point $x\in\cap_{n=0}^{\infty}\mathrm{Fix}U_{n}$,
and is bounded perturbations resilient with respect to $\cap_{n=0}^{\infty}\mathrm{Fix}U_{n}$.
\begin{thm}
\label{main_res}Let $\left\{ U_{n}\right\} _{n=0}^{\infty}$ be a
sequence of $2^{-1}$-strongly quasi-nonexpansive operators such that
$\cap_{n=0}^{\infty}\mathrm{Fix}U_{n}\not=\emptyset$ and let $\rho\in\left[0,1\right]$.
Assume that the sequence $\left\{ N_{k}\right\} _{k=0}^{\infty}$
is bounded. Suppose further that for each $k\in\mathbb{N}$ and each
$n\in c_{k}^{-1}\left(\left\{ 0\right\} \right)$, in addition, $U_{j_{n}^{k}}$
is a cutter, or that $\alpha_{k}\left(n\right)=1$. Let $\left\{ T_{k}\right\} _{k=0}^{\infty}$
be a sequence of operators generated by the GMSA procedure as in \eqref{eq:-3},
and assume that $T_{k}$ is $\rho$-strongly quasi-nonexpansive operator
for each $k\in\mathbb{N}$. Let $\left\{ \lambda_{k}\right\} _{k=0}^{\infty}\subset\left[\varepsilon,1+\rho-\varepsilon\right]$
(where $\varepsilon\in\left(0,1\right]$) be a sequence. Then the
following statements hold:
\begin{enumerate}
\item For each $x^{0}\in\mathcal{H}$, the sequence $\left\{ x^{k}\right\} _{k=0}^{\infty}$
generated by Algorithm \ref{alg=0000201} with respect to the sequences
$\left\{ T_{k}\right\} _{k=0}^{\infty}$ and $\left\{ \lambda_{k}\right\} _{k=0}^{\infty}$
is strongly Fej{\'e}r monotone with respect to $\cap_{n=0}^{\infty}\mathrm{Fix}U_{n}$.
Namely, 
\[
\left\Vert x^{k+1}-z\right\Vert ^{2}\le\left\Vert x^{k}-z\right\Vert ^{2}-\varepsilon\left(1+\rho-\varepsilon\right)^{-1}\left\Vert \left(x^{k+1}-x^{k}\right)\right\Vert ^{2}
\]
for each $k\in\mathbb{N}$, and as a result, $\left\Vert x^{k+1}-x^{k}\right\Vert \rightarrow0$.
\item If, in addition, the sequence $\left\{ I_{N_{k}}^{k}\right\} _{k=0}^{\infty}$
satisfies Equation \eqref{eq:-16}, the family $\left\{ \mathrm{Fix}U_{n}\right\} _{n=0}^{\infty}$
is boundedly regular and the operator $U_{n}$ is approximately shrinking
for each $n\in\mathbb{N}$, then for each $x^{0}\in\mathcal{H}$,
the sequence $\left\{ x^{k}\right\} _{k=0}^{\infty}$ generated by
Algorithm \ref{alg=0000201} with respect to the sequences $\left\{ T_{k}\right\} _{k=0}^{\infty}$
and $\left\{ \lambda_{k}\right\} _{k=0}^{\infty}$ converges in the
norm of $\mathcal{H}$ to a point $x\in\cap_{n=0}^{\infty}\mathrm{Fix}U_{n}$.
\item Under the assumptions of (ii), if for each $k\in\mathbb{N}$, the
$\lambda_{k}$-relaxation of the operator $T_{k}$ is nonexpansive,
then Algorithm \ref{alg=0000201} with respect to the sequences $\left\{ T_{k}\right\} _{k=0}^{\infty}$
and $\left\{ \lambda_{k}\right\} _{k=0}^{\infty}$ is bounded perturbations
resilient with respect to $\cap_{n=0}^{\infty}\mathrm{Fix}U_{n}$.
\end{enumerate}
\end{thm}
\begin{proof}
For each $k\in\mathbb{N}$, denote by $S_{k}$ the $\lambda_{k}$-relaxation
of $T_{k}$. \\
\textit{(i)} Let $x^{0}\in\mathcal{H}$ and let $\left\{ x^{k}\right\} _{k=0}^{\infty}$
be the sequence generated by Algorithm \ref{alg=0000201} with respect
to the sequences $\left\{ T_{k}\right\} _{k=0}^{\infty}$ and $\left\{ \lambda_{k}\right\} _{k=0}^{\infty}$.
Assume that $k\in\mathbb{N}$ and that $z\in\cap_{n=0}^{\infty}\mathrm{Fix}U_{n}$.
By Corollary \ref{cor:2.7}, the $2^{-1}\left(1+\rho\right)$-relaxation
of $T_{k}$ is a cutter, that is, is $1$-strongly quasi-nonexpansive.
Therefore, its $2\lambda_{k}\left(1+\rho\right)^{-1}$-relaxation,
which is the $\lambda_{k}$-relaxation of $T_{k}$, is $\left(1+\rho-\lambda_{k}\right)\lambda_{k}^{-1}$-strongly
quasi-nonexpansive by Theorem \ref{thm:2.1.39}. It follows that
\begin{align}
\left\Vert x^{k+1}-z\right\Vert ^{2} & =\left\Vert S_{k}x^{k}-z\right\Vert ^{2}\le\left\Vert x^{k}-z\right\Vert ^{2}-\left(1+\rho-\lambda_{k}\right)\lambda_{k}^{-1}\left\Vert S_{k}x^{k}-x^{k}\right\Vert ^{2}\nonumber \\
 & =\left\Vert x^{k}-z\right\Vert ^{2}-\left(1+\rho-\lambda_{k}\right)\lambda_{k}^{-1}\left\Vert x^{k+1}-x^{k}\right\Vert ^{2}\nonumber \\
 & \le\left\Vert x^{k}-z\right\Vert ^{2}-\varepsilon\left(1+\rho-\varepsilon\right)^{-1}\left\Vert x^{k+1}-x^{k}\right\Vert ^{2}.\label{eq:-1-2}
\end{align}
Since $\varepsilon\left(1+\rho-\varepsilon\right)^{-1}>0$, we see
from \eqref{eq:-1-2} that $\left\{ x^{k}\right\} _{k=0}^{\infty}$
is strongly Fej{\'e}r monotone with respect to $\cap_{n=0}^{\infty}\mathrm{Fix}U_{n}$.
Since the real sequence $\left\{ \left\Vert x^{k}-z\right\Vert ^{2}\right\} _{k=0}^{\infty}$
is monotonically decreasing and bounded from below by $0$, it converges
and, by \eqref{eq:-1-2}, $\left\Vert x^{k+1}-x^{k}\right\Vert \rightarrow0$.

\textit{(ii)} Let $x^{0}\in\mathcal{H}$ and let $\left\{ x^{k}\right\} _{k=0}^{\infty}$
be the sequence generated by Algorithm \ref{alg=0000201} with respect
to the sequences $\left\{ T_{k}\right\} _{k=0}^{\infty}$ and $\left\{ \lambda_{k}\right\} _{k=0}^{\infty}$.
Assume that the sequence $\left\{ I_{N_{K}}^{k}\right\} _{k=0}^{\infty}$
satisfies Equation \eqref{eq:-16}, that the family $\left\{ \mathrm{Fix}U_{n}\right\} _{n=0}^{\infty}$
is boundedly regular and that the operator $U_{n}$ is approximately
shrinking for each $n\in\mathbb{N}$. Let $n\in\mathbb{N}$ be a natural
number. Clearly, there is a sequence $\left\{ \ell_{k}\right\} _{k=0}^{\infty}\subset\left\{ 0,1,\dots,M_{n}-1\right\} $
such that $n\in I_{N_{k+\ell_{k}}}^{k+\ell_{k}}$ for each $k\in\mathbb{N}$.
We have
\begin{align}
\lambda_{k+\ell_{k}}\left(T_{k+\ell_{k}}\left(x^{k+\ell_{k}}\right)-x^{k+\ell_{k}}\right) & =S_{k+\ell_{k}}\left(x^{k+\ell_{k}}\right)-x^{k+\ell_{k}}=x^{k+\ell_{k}+1}-x^{k+\ell_{k}}\label{eq:-20}
\end{align}
for each $k\in\mathbb{N}$. Since $\left\{ \lambda_{k+\ell_{k}}\right\} _{k=0}^{\infty}\subset\left[\varepsilon,1+\rho-\varepsilon\right]$,
we obtain, from \eqref{eq:-20} and \textit{(i),} that 
\[
\lim_{k\rightarrow\infty}\left\Vert T_{k+\ell_{k}}\left(x^{k+\ell_{k}}\right)-x^{k+\ell_{k}}\right\Vert =0.
\]
 The sequence $\left\{ x^{k+\ell_{k}}\right\} _{k=0}^{\infty}$ is
bounded because it is Fej{\'e}r monotone with respect to $\cap_{n=0}^{\infty}\mathrm{Fix}U_{n}$
by \textit{(i)}. Hence, by Lemmata \ref{lem1} and \ref{lem5}, 
\begin{equation}
\lim_{k\rightarrow\infty}\max_{i\in I_{N_{k+\ell_{k}}}^{k+\ell_{k}}}d\left(x^{k+\ell_{k}},\mathrm{Fix}U_{i}\right)=0.\label{eq:-22}
\end{equation}
Now for each $k\in\mathbb{N}$,

\begin{equation}
\left\Vert x^{k+\ell_{k}}-x^{k}\right\Vert =\left\Vert \sum_{i=0}^{\ell_{k}-1}\left(x^{k+i+1}-x^{k+i}\right)\right\Vert \le\sum_{i=0}^{\ell_{k}-1}\left\Vert \left(x^{k+i+1}-x^{k+i}\right)\right\Vert .\label{eq:-21}
\end{equation}
Due to the finite number of summands in \eqref{eq:-21} we deduce
from \textit{(i)} that 
\begin{equation}
\lim_{k\rightarrow\infty}\left\Vert x^{k+\ell_{k}}-x^{k}\right\Vert =0.\label{eq:-23}
\end{equation}
Since $n\in I_{N_{k+\ell_{k}}}^{k+\ell_{k}}$ for each $k\in\mathbb{N}$
and $\mathrm{Fix}U_{n}\not=\emptyset$, we see, by \eqref{eq:-23},
Remark \ref{fix_cl_cnv}, the definition of the metric projection
and the triangle inequality, that for each $k\in\mathbb{N}$,
\begin{align}
d\left(x^{k},\mathrm{Fix}U_{n}\right) & =\left\Vert x^{k}-P_{\mathrm{Fix}U_{n}}\left(x^{k}\right)\right\Vert \le\left\Vert x^{k}-P_{\mathrm{Fix}U_{n}}\left(x^{k+\ell_{k}}\right)\right\Vert \nonumber \\
 & \le\left\Vert x^{k+\ell_{k}}-x^{k}\right\Vert +\left\Vert x^{k+\ell_{k}}-P_{\mathrm{Fix}U_{n}}\left(x^{k+\ell_{k}}\right)\right\Vert \nonumber \\
 & =\left\Vert x^{k+\ell_{k}}-x^{k}\right\Vert +d\left(x^{k+\ell_{k}},\mathrm{Fix}U_{n}\right)\nonumber \\
 & \le\left\Vert x^{k+\ell_{k}}-x^{k}\right\Vert +\max_{i\in I_{N_{k+\ell_{k}}}^{k+\ell_{k}}}d\left(x^{k+\ell_{k}},\mathrm{Fix}U_{i}\right).\label{eq:-5}
\end{align}
Combining \eqref{eq:-5} with \eqref{eq:-22} and \eqref{eq:-23},
we obtain that $\lim_{k\rightarrow\infty}d\left(x^{k},\mathrm{Fix}U_{n}\right)=0$.
Since $n\in\mathbb{N}$ is arbitrary, the bounded regularity of the
family $\left\{ \mathrm{Fix}U_{n}\right\} _{n=0}^{\infty}$ implies
that
\begin{equation}
\lim_{k\rightarrow\infty}d\left(x^{k},\cap_{n=0}^{\infty}\mathrm{Fix}U_{n}\right)=0.\label{eq:-25}
\end{equation}
It now follows, from \eqref{eq:-25}, Remark \ref{fix_cl_cnv} and
Theorem \ref{Fejer}, that the sequence $\left\{ x^{k}\right\} _{k=0}^{\infty}$
converges in the norm of $\mathcal{H}$ to a point $x\in\cap_{n=0}^{\infty}\mathrm{Fix}U_{n}$.

\textit{(iii)} Assume that the assumptions of \textit{(ii)} hold and
that the $\lambda_{k}$-relaxation of the operator $T_{k}$ is nonexpansive.
Define $C:=\cap_{n=0}^{\infty}\mathrm{Fix}U_{n}$. Let $\left\{ \beta_{k}\right\} _{k=0}^{\infty}$
be a sequence of positive real numbers such that $\sum_{k=0}^{\infty}\beta_{k}<\infty$
and let $\{v^{k}\}_{k=0}^{\infty}$ be a bounded sequence in $\mathcal{H}$.
Suppose that $y^{0}\in\mathcal{H}$ and consider the sequence $\left\{ y^{k}\right\} _{k=0}^{\infty}$
generated by the iterative process $y^{k+1}:=S_{k}(y^{k}+\beta_{k}v^{k})$.
Let $q\in\mathbb{N}$ and $y\in\mathcal{H}$ be arbitrary. For each
$k\in\mathbb{N}$, Define $\gamma_{k}:=\beta_{k}\left\Vert v^{k}\right\Vert \in\left[0,\infty\right)$
and $\ell_{k}:=q+k$. Define $x^{0}:=y$. Clearly, $\sum_{k=0}^{\infty}\gamma_{k}<\infty$
and the sequence $\left\{ I_{N_{\ell_{k}}}^{\ell_{k}}\right\} _{k=0}^{\infty}$
satisfies Equation \eqref{eq:-16}. By \ref{lem1} and \textit{(ii)},
the sequence $\left\{ x^{k}\right\} _{k=0}^{\infty}$ generated by
Algorithm \ref{alg=0000201} with respect to the sequences $\left\{ T_{\ell_{k}}\right\} _{k=0}^{\infty}$
and $\left\{ \lambda_{\ell_{k}}\right\} _{k=0}^{\infty}$ converges
in the norm of $\mathcal{H}$ to a point in $C$. Then, for each $k\in\mathbb{N}$,
we have 
\[
x^{k+1}=S_{\ell_{k}}\left(x^{k}\right)=S_{\ell_{k}}\cdots S_{\ell_{1}}S_{\ell_{0}}\left(y\right)=S_{q+k}\cdots S_{q+1}S_{q}\left(y\right)
\]
and hence the sequence $\left\{ S_{q+k}\cdots S_{q+1}S_{q}\left(y\right)\right\} _{k=0}^{\infty}$
converges in the norm of $\mathcal{H}$ to an element of $C$ for
an arbitrary $y\in\mathcal{H}$. Since for each $k\in\mathbb{N}$,
the operator $S_{q+k}$ is nonexpansive, we obtain
\begin{align*}
\left\Vert y^{k+1}-S_{k}\left(y^{k}\right)\right\Vert  & =\left\Vert S_{k}(y^{k}+\beta_{k}v^{k})-S_{k}\left(y^{k}\right)\right\Vert \le\beta_{k}\left\Vert v^{k}\right\Vert =\gamma_{k}.
\end{align*}
By Lemma \ref{lem3}, Equation \eqref{eq:-16} and Remark \ref{fix_cl_cnv},
$C=\cap_{k=0}^{\infty}\mathrm{Fix}T_{k}=\cap_{k=0}^{\infty}\mathrm{Fix}S_{k}$.
We now deduce from Theorem \ref{error} along with Remark \ref{fix_cl_cnv}
that the sequence $\left\{ y^{k}\right\} _{k=0}^{\infty}$ converges
in the norm of $\mathcal{H}$ to an element of $C$ as well, proving
that the sequence $\left\{ x^{k}\right\} _{k=0}^{\infty}$ is bounded
perturbations resilient with respect to $C$.
\end{proof}
\begin{cor}
Let $\left\{ U_{n}\right\} _{n=0}^{\infty}$ be a sequence of $2^{-1}$-strongly
quasi-nonexpansive operators such that $\cap_{n=0}^{\infty}\mathrm{Fix}U_{n}\not=\emptyset$.
Suppose further that for each $k\in\mathbb{N}$ and each $n\in c_{k}^{-1}\left(\left\{ 0\right\} \right)$,
$U_{j_{n}^{k}}$ is a cutter, or that $\alpha_{k}\left(n\right)=1$.
Assume that the sequences $\left\{ N_{k}\right\} _{k=0}^{\infty}$
and $\left\{ \left\{ P_{n}^{k}\right\} _{n\in\mathbb{N}\backslash\left\{ 0\right\} \cap L_{N_{k}}}\right\} _{k=0}^{\infty}$
are bounded. Define $\rho:=\frac{\varepsilon}{2M^{K}}$, where $K:=\max_{k\in\mathbb{N}}N_{k}$
and $M:=\max_{k\in\mathbb{N}}\left\{ P_{n}^{k}\right\} _{n\in\mathbb{N}\backslash\left\{ 0\right\} \cap L_{N_{k}}}$.
Let $\left\{ T_{k}\right\} _{k=0}^{\infty}$ be a sequence generated
by the GMSA procedure as in \eqref{eq:-3}, and let $\left\{ \lambda_{k}\right\} _{k=0}^{\infty}\subset\left[1+\rho-\varepsilon\right]$
(where $\varepsilon\in\left(0,1\right]$) be a sequence. Then, given
$x^{0}\in\mathcal{H}$, the sequence $\left\{ x^{k}\right\} _{k=0}^{\infty}$
generated by Algorithm \ref{alg=0000201} with respect to the sequences
$\left\{ T_{k}\right\} _{k=0}^{\infty}$ and $\left\{ \lambda_{k}\right\} _{k=0}^{\infty}$
satisfies all the statements of Theorem \ref{main_res}.
\end{cor}
\begin{proof}
Clearly, $\rho\in\left[0,1\right]$. By Lemma \ref{lem3}, the operator
$T_{k}$ is $\frac{\varepsilon}{2\Pi_{i=1}^{N_{k}}P_{i}^{k}}$-strongly
quasi-nonexpansive and hence $\rho$- strongly nonexpansive for each
$k\in\mathbb{N}$.

The result now follows from Theorem \ref{main_res}.
\end{proof}

\subsection{\protect\label{firm_nonex}The method based on relaxations of firmly
nonexpansive operators}

In this subsection we consider the particular case of $2^{-1}$-strongly
quasi-nonexpansive input operators, wherein they are $2^{-1}$-firmly
quasi-nonexpansive. In this case we establish the non-expansiveness
of the $\lambda_{k}$-relaxations of the operators $\left\{ T_{k}\right\} _{k=0}^{\infty}$
from Theorem \ref{main_res}\textit{(iii)},\textit{ }as the following
lemma\textit{ }shows\textit{.}
\begin{lem}
\label{outpt_nonexp}Let $\left\{ T_{k}\right\} _{k=0}^{\infty}$
be a sequence of operators generated by the GMSA procedure as in \eqref{eq:-3},
let $\rho\in\left[0,1\right]$ and assume that $T_{k}$ is a $\rho$-firmly
nonexpansive operator for each $k\in\mathbb{N}$. Let $\varepsilon\in\left(0,1\right]$
and let $\left\{ \lambda_{k}\right\} _{k=0}^{\infty}\subset\left[\varepsilon,1+\rho-\varepsilon\right]$
be a sequence. Then for each $k\in\mathbb{N}$, the $\lambda_{k}$-relaxation
of $T_{k}$, $T_{k\lambda_{k}}$ is nonexpansive.
\end{lem}
\begin{proof}
Let $k\in\mathbb{N}$. By Corollary \ref{cor:=000020frm-nexp}, the
$2^{-1}\left(1+\rho\right)$-relaxation of $T_{k}$ is firmly nonexpansive,
i.e., it is $1$-firmly nonexpansive. Therefore, its $2\lambda_{k}\left(1+\rho\right)^{-1}$-relaxation,
which is $T_{k\lambda_{k}}$, is $\left(1+\rho-\lambda_{k}\right)\lambda_{k}^{-1}$-firmly
nonexpansive, by Theorem \ref{FNE_relax_FNE}. So, by Remark \ref{FNE-SQNE},
it is nonexpansive.
\end{proof}
If we assume that for each $n\in\mathbb{N}$, the input operator $U_{n}$
of the GMSA procedure is $2^{-1}$-firmly nonexpansive (or, equivalently,
$4\cdot3^{-1}$-relaxed firmly nonexpansive by Theorem \ref{FNE_relax_FNE}),
then we can guarantee in Theorem \ref{main_res}\textit{(iii)} that
the output operators $\left\{ T_{k}\right\} _{k=0}^{\infty}$ are
$\rho$-firmly nonexpansive and hence $\rho$-strongly quasi-nonexpansive
for some $\rho\in\left[0,1\right]$. To this end we need the following
auxiliary lemma.
\begin{lem}
\label{lem4}Let $\varepsilon\in\left(0,1\right]$, let $\left\{ U_{n}\right\} _{n=0}^{\infty}$
be a sequence of $2^{-1}$-firmly nonexpansive operators, let $k\in\mathbb{N}$
and let $n\in L_{N_{k}}$. Then the following assertions hold:
\begin{enumerate}
\item If $n\in c_{k}^{-1}\left(\left\{ 0\right\} \right)$ and $U_{-j_{n}^{k}}$
is, in addition, a firmly nonexpansive operator, or if $\alpha_{k}\left(n\right)=1$,
then the intermediate module $V_{n}^{k}$ generated by the GMSA procedure
is $\frac{\varepsilon}{2\Pi_{i=1}^{n}P_{i}^{k}}$-firmly nonexpansive
and hence nonexpansive.
\item If $n\in c_{k}^{-1}\left(\left\{ 1\right\} \right)\cup c_{k}^{-1}\left(\left\{ 2\right\} \right)$,
then the intermediate module $V_{n}^{k}$ generated by the GMSA procedure
is $\frac{\varepsilon}{2\Pi_{i=1}^{n}P_{i}^{k}}$-firmly nonexpansive
and hence nonexpansive.
\end{enumerate}
We use the convention that for a non-positive integer $n$, the empty
product $\Pi_{i=1}^{n}P_{i}^{k}=1$.
\end{lem}
\begin{proof}
If $n$ is non-positive then, by \eqref{eq:-3}, $V_{n}^{k}=U_{-n}$
and hence $V_{n}^{k}$ is $2^{-1}$-firmly nonexpansive, which implies
that it is $\frac{\varepsilon}{2\Pi_{i=1}^{n}P_{i}^{k}}$-firmly nonexpansive.
The proof proceeds by induction on the set $\mathbb{N}\backslash\left\{ 0\right\} \cap L_{N_{k}}$.
Let $n$ be a positive integer in $L_{N_{k}}$ and assume that the
statement of the lemma holds for each positive integer $j<n$ in $\mathbb{N}\backslash\left\{ 0\right\} \cap L_{N_{k}}$.
We show that it also holds for $n$ is this case. We have the following
three cases:

\textbf{\textit{Case 1:}}\textit{ }$c_{k}\left(n\right)=0$. By \eqref{eq:-3},
\[
V_{n}^{k}=Id+\alpha_{k}\left(n\right)\left(V_{j_{n}^{k}}^{k}-Id\right)=Id+\alpha_{k}\left(n\right)\left(U_{-j_{n}^{k}}-Id\right),
\]
where $\alpha_{k}\left(n\right)\in\left[\varepsilon,2-\varepsilon\right]$
and $j_{n}^{k}\in J_{n}^{k}\subset L_{0}$. Clearly, $J_{n}^{k}=\left\{ j_{n}^{k}\right\} $
and, by \eqref{eq:},
\[
I_{n}^{k}=\cup_{j\in J_{n}^{k}}I_{j}^{k}=I_{j_{n}^{k}}^{k}=\left\{ -j_{n}^{k}\right\} .
\]
Now, if $U_{-j_{n}^{k}}$ is, additionally, firmly nonexpansive then,
by Theorem \ref{FNE_relax_FNE}, the operator $V_{n}^{k}$ is $\frac{2-\alpha_{k}\left(n\right)}{\alpha_{k}\left(n\right)}$-firmly
nonexpansive. The inequality
\[
\frac{2-\alpha_{k}\left(n\right)}{\alpha_{k}\left(n\right)}\ge\frac{\varepsilon}{2-\varepsilon}>\frac{\varepsilon}{2}\ge\frac{\varepsilon}{2\Pi_{i=1}^{n}P_{i}^{k}}
\]
yields that $V_{n}^{k}$ is $\frac{\varepsilon}{2\Pi_{i=1}^{n}P_{i}^{k}}$-firmly
nonexpansive. If $\alpha_{k}\left(n\right)=1$, then $V_{n}^{k}=U_{-j_{n}^{k}}$
is $2^{-1}$-firmly nonexpansive, which implies that it is $\frac{\varepsilon}{2\Pi_{i=1}^{n}P_{i}^{k}}$-firmly
nonexpansive and hence nonexpansive.

\textbf{\textit{Case 2:}} $c_{k}\left(n\right)=1$. In this case (by
\eqref{eq:-3}) $V_{n}^{k}=\sum_{j\in J_{n}^{k}}\omega_{n}^{k}\left(j\right)V_{j}^{k}$,
where $J_{n}^{k}\subset L_{n-1}$ and the weights $\left\{ \omega_{n}^{k}\left(j\right)\right\} _{j\in J_{n}^{k}}$
are in the interval $\left[\varepsilon,1\right]$. By the the induction
hypothesis and since $U_{n}$ is $2^{-1}$-firmly nonexpansive for
each $n\in\mathbb{N}$, the operator $V_{j}^{k}$ is $\frac{\varepsilon}{2\Pi_{i=1}^{j}P_{i}^{k}}$-firmly
nonexpansive for each $j\in J_{n}^{k}$. By Corollary \ref{comp_conv_comb_firm}\textit{(i)},
$V_{n}^{k}$ is $\min_{j\in J_{n}^{k}}\left\{ \frac{\varepsilon}{2\Pi_{i=1}^{j}P_{i}^{k}}\right\} $-firmly
nonexpansive. Since $\min_{j\in J_{n}^{k}}\left\{ \frac{\varepsilon}{2\Pi_{i=1}^{j}P_{i}^{k}}\right\} \ge\frac{\varepsilon}{2\Pi_{i=0}^{n}P_{i}^{k}}$,
it follows that $V_{n}^{k}$ is $\frac{\varepsilon}{2\Pi_{i=1}^{n}P_{i}^{k}}$-firmly
nonexpansive and hence nonexpansive also in this case.

\textbf{\textit{Case 3:}}\textit{ }$c_{k}\left(n\right)=2$. In this
case (by \eqref{eq:-3}) $V_{n}^{k}=V_{o_{n}^{k}\left(P_{n}^{k}\right)}^{k}\cdots V_{o_{n}^{k}\left(1\right)}^{k}$,
where $o_{n}^{k}$ is an order function having its values in $J_{n}^{k}$.
Similarly to \textbf{Case 2}, by Corollary \ref{comp_conv_comb_firm}\textit{(ii)},
$V_{n}^{k}$ is $\frac{\min_{j\in J_{n}^{k}}\left\{ \frac{\varepsilon}{2\Pi_{i=1}^{j}P_{i}^{k}}\right\} }{P_{n}^{k}}$-firmly
nonexpansive, which implies that it is $\frac{\varepsilon}{2\Pi_{i=1}^{n}P_{i}^{k}}$-firmly
nonexpansive and hence nonexpansive.
\end{proof}
\begin{cor}
\label{cor:firm_nonex}Let $\left\{ U_{n}\right\} _{n=0}^{\infty}$
be a sequence of approximately shrinking and $\gamma_{n}$-relaxed
firmly nonexpansive operators, where $\gamma_{n}\in\left(0,4\cdot3^{-1}\right]$
for each $n\in\mathbb{N}$, such that the family $\left\{ \mathrm{Fix}U_{n}\right\} _{n=0}^{\infty}$
is boundedly regular and $\cap_{n=0}^{\infty}\mathrm{Fix}U_{n}\not=\emptyset$.
Suppose that for each $k\in\mathbb{N}$ and each $n\in c_{k}^{-1}\left(\left\{ 0\right\} \right)$,
in addition, $U_{j_{n}^{k}}$ is firmly nonexpansive, or that $\alpha_{k}\left(n\right)=1$.
Suppose that the sequence $\left\{ N_{k}\right\} _{k=0}^{\infty}$
is bounded. Let $\left\{ T_{k}\right\} _{k=0}^{\infty}$ be a sequence
generated by the GMSA procedure as in \eqref{eq:-3}. Assume that
the sequence $\left\{ I_{N_{k}}^{k}\right\} _{k=0}^{\infty}$ satisfies
Equation \eqref{eq:-16}. Let $x^{0}\in\mathcal{H}$. Then:
\begin{enumerate}
\item If there exists $\rho\in\left[0,1\right]$ such that for each $k\in\mathbb{N}$,
the operator $T_{k}$ is $\rho$-firmly nonexpansive, then the sequence
$\left\{ x^{k}\right\} _{k=0}^{\infty}$ generated by Algorithm \ref{alg=0000201}
with respect to the sequences $\left\{ T_{k}\right\} _{k=0}^{\infty}$
and $\left\{ \lambda_{k}\right\} _{k=0}^{\infty}\subset\left[1+\rho-\varepsilon\right]$
is strongly Fej{\'e}r monotone with respect to $\cap_{n=0}^{\infty}\mathrm{Fix}U_{n}$,
converges in the norm of $\mathcal{H}$ to a point $x\in\cap_{n=0}^{\infty}\mathrm{Fix}U_{n}$
and is bounded perturbations resilient with respect to $\cap_{n=0}^{\infty}\mathrm{Fix}U_{n}$.
\item If the sequence $\left\{ \left\{ P_{n}^{k}\right\} _{n\in\mathbb{N}\backslash\left\{ 0\right\} \cap L_{N_{k}}}\right\} _{k=0}^{\infty}$
is also bounded, then there exists $\rho\in\left[0,1\right]$ such
that the operator $T_{k}$ is $\rho$-firmly nonexpansive for each
$k\in\mathbb{N}$. Consequently, by (i) above, the sequence $\left\{ x^{k}\right\} _{k=0}^{\infty}$
generated by Algorithm \ref{alg=0000201} with respect to the sequences
$\left\{ T_{k}\right\} _{k=0}^{\infty}$ and $\left\{ \lambda_{k}\right\} _{k=0}^{\infty}\subset\left[1+\rho-\varepsilon\right]$
is strongly Fej{\'e}r monotone with respect to $\cap_{n=0}^{\infty}\mathrm{Fix}U_{n}$,
converges in the norm of $\mathcal{H}$ to a point $x\in\cap_{n=0}^{\infty}\mathrm{Fix}U_{n}$
and is bounded perturbations resilient with respect to $\cap_{n=0}^{\infty}\mathrm{Fix}U_{n}$.
\end{enumerate}
\end{cor}
\begin{proof}
\textit{(i)} By Theorem \ref{FNE_relax_FNE}, for each $n\in\mathbb{N}$,
the operator $U_{n}$ is $\left(2-\gamma_{n}\right)\gamma_{n}$-firmly
nonexpansive and hence $2^{-1}$-firmly nonexpansive. By Remark \ref{FNE-SQNE},
the operator $U_{n}$ is $2^{-1}$-strongly quasi-nonexpansive for
each $n\in\mathbb{N}$ and the operator $T_{k}$ is $\rho$-strongly
quasi nonexpansive for each $k\in\mathbb{N}$. Moreover, by Theorem
\ref{thm:2.2.5}, for each $n\in c_{k}^{-1}\left(\left\{ 0\right\} \right)$,
$U_{j_{n}^{k}}$ is, additionally, a nonexpansive cutter, or $\alpha_{k}\left(n\right)=1$.
The result now follows from Theorem \ref{main_res} along with Lemma
\ref{outpt_nonexp}.

\textit{(ii)} Define $\rho:=\frac{\varepsilon}{2M^{K}}$, where $K:=\max_{k\in\mathbb{N}}N_{k}$
and $M:=\max_{k\in\mathbb{N}}\left\{ P_{n}^{k}\right\} _{n\in\mathbb{N}\backslash\left\{ 0\right\} \cap L_{N_{k}}}$.
By Lemma \ref{lem4}, the operator $T_{k}$ is $\frac{\varepsilon}{2\Pi_{i=1}^{n}P_{i}^{k}}$-firmly
nonexpansive and hence $\rho$-firmly nonexpansive for each $k\in\mathbb{N}$.
\end{proof}
\begin{example}
For each $n\in\mathbb{N}$, let $C_{n}$ be an arbitrary finite-dimensional
vector subspace of $\mathcal{H}$. Then by Proposition \ref{bounded_reg},
the family $\left\{ C_{n}\right\} _{n\in\mathbb{N}}$ is a boundedly
regular family of nonempty, closed and convex subsets of $\mathcal{H}$
with the nonempty intersection which contains $\left\{ 0\right\} $,
since each $C_{n}$ is a finite-dimensional normed linear space and
hence is locally compact. For each $n\in\mathbb{N}$, define $U_{n}:=P_{C_{n}}$.
By Examples \eqref{ex_as} and \eqref{metric=000020projection}, the
operator $U_{n}$ is approximately shrinking and firmly nonexpansive
(that is, $1$-firmly nonexpansive) for each $n\in\mathbb{N}$. We
see that the family $\left\{ \mathrm{Fix}U_{n}\right\} _{n=0}^{\infty}$
is boundedly regular and $\cap_{n=0}^{\infty}\mathrm{Fix}U_{n}\not=\emptyset$.
Choose the sequences $\left\{ N_{k}\right\} _{k=0}^{\infty}$, $\left\{ \left\{ P_{n}^{k}\right\} _{n\in\mathbb{N}\backslash\left\{ 0\right\} \cap L_{N_{k}}}\right\} _{k=0}^{\infty}$
and $\left\{ I_{N_{k}}^{k}\right\} _{k=0}^{\infty}$ and the rest
of parameters of GMSA procedure as in Example \ref{Concrete=000020example}.
Clearly, the sequences $\left\{ N_{k}\right\} _{k=0}^{\infty}$ and
$\left\{ \left\{ P_{n}^{k}\right\} _{n\in\mathbb{N}\backslash\left\{ 0\right\} \cap L_{N_{k}}}\right\} _{k=0}^{\infty}$
are bounded and the sequence $\left\{ I_{N_{k}}^{k}\right\} _{k=0}^{\infty}$
satisfies Equation \eqref{eq:-16}.

These settings fit the framework of Corollary \ref{cor:firm_nonex}\textit{(ii)}
above and provide an example of a strongly convergent and perturbation
resilient method, particularly, in case where $\mathcal{H}$ is an
infinite-dimensional space.
\end{example}

\section{\protect\label{SM}The superiorization methodology}

In many scientific or real-world problems that are modeled as constrained
minimization problems, aiming for the (constrained) optimal solution
may require high prices for the investment of time, energy, and resources.

The Superiorization Methodology (SM) addresses this by introducing
low-cost perturbations into a feasibility-seeking algorithm. The resulting
algorithm retains the convergence to a feasible point, while steering
the iterates to a ``superior'' feasible point. In this context,
a superior feasible point has a reduced (but not necessarily optimal)
value of the associated objective function.

Additional information and references on the superiorization methodology
are available in the papers listed in the bibliographic collection
on the dedicated Webpage \cite{SM-bib-page}. For recent works containing
introductory material on the SM see, for example, \cite{asymmetric-2023},
\cite{erturk-salim-2023}, \cite{humphries-2022} and \cite{Torregrosa-2024}.

We consider in the sequel the following algorithm which provides a
general framework for superiorization methods with negative (sub)gradient
perturbation.
\begin{lyxalgorithm}
\label{alg=0000202}Given $y^{0}\in\mathcal{H}$, $\varepsilon\in\left(0,1\right]$,
$\phi:\mathcal{H}\rightarrow\mathbb{R}$ a convex and continuous function,
a sequence $\left\{ \mathcal{T}_{k}\right\} _{k=0}^{\infty}$ of operators,
a sequence $\left\{ M_{k}\right\} _{k=0}^{\infty}$ of positive integers
and a family of positive real sequences $\left\{ \left\{ \beta_{k,n}\right\} _{n=1}^{M_{k}}\right\} _{k=0}^{\infty}$
such that $\sum_{k=0}^{\infty}\sum_{n=1}^{M_{k}}\beta_{k,n}<\infty$,
the algorithm is defined by the recurrences
\[
y^{k+1}:=\mathcal{T}_{k}\left(y^{k}+\sum_{n=1}^{M_{k}}\beta_{k,n}v^{k,n}\right)
\]
\[
\mathrm{wherein}
\]
\begin{equation}
v^{k,n+1}:=\begin{cases}
-\left\Vert s^{k,n}\right\Vert ^{-1}s^{k,n}, & \mathrm{if}\,\,0\not\in\partial\phi\left(y^{k}+\sum_{i=1}^{n}\beta_{k,i}v^{k,i}\right),\\
0, & \mathrm{if}\,\,0\in\partial\phi\left(y^{k}+\sum_{i=1}^{n}\beta_{k,i}v^{k,i}\right),
\end{cases}\label{eq:-19}
\end{equation}
for each $k\in\mathbb{N}$ and each $n=0,1,\dots M_{k}-1$, where
$s^{k,n}$ is a selection of the subgradient $\partial\phi\left(y^{k}+\sum_{i=1}^{n}\beta_{k,i}v^{k,i}\right)$
(which exists by Theorem \ref{subdiff_ne+b}) for each $k\in\mathbb{N}$
and each $n=0,1,\dots M_{k}-1$ (recalling that, by definition, $\sum_{i=1}^{0}\beta_{k,i}v^{k,i}=0$).
\end{lyxalgorithm}
This formulation of the algorithm generalizes, in more than one way,
\cite[Algorithm 4.1]{CZ_sup} as will be clarified below.

We need the following auxiliary lemma to further investigate the behavior
of Algorithm \ref{alg=0000202}.
\begin{lem}
[{\cite[Lemma 5.3]{GDSA_inconsist}}]\label{aux_lem} Let $\phi:\mathcal{H}\rightarrow\mathbb{R}$
be a convex and continuous real-valued objective function. For an
arbitrary nonempty subset $C$ of $\mathcal{H}$ and $y,z\in C$ such
that $z\in\mathrm{Argmin}_{x\in C}\phi\left(x\right)$ and $y\not\in\mathrm{Argmin}_{x\in C}\phi\left(x\right)$,
there exist real numbers $r_{1}>0$ and $r_{2}>0$ so that for each
$\overline{y}\in B\left(y,r_{1}\right)$ and $v\in\partial\phi\left(\overline{y}\right)$,
the following assertions are satisfied:
\begin{enumerate}
\item $0\not\in\partial\phi\left(\overline{y}\right)$ and for each $\overline{z}\in B\left(z,r_{2}\right)$
\[
\left\langle \left\Vert v\right\Vert ^{-1}v,\overline{z}-\overline{y}\right\rangle <0.
\]
\item We have 
\[
\left\langle \left\Vert v\right\Vert ^{-1}v,z-\overline{y}\right\rangle <-2^{-1}r_{2}.
\]
\item Let $p$ be a nonnegative integer. Assume that $\left\{ \alpha_{n}\right\} _{n=1}^{p}$
is a sequence of positive real numbers such that $\sum_{n=1}^{p}a_{n}<2^{-1}r_{1}$
and $\left\{ v^{n}\right\} _{n=1}^{p}\subset\mathcal{H}\backslash\left\{ 0\right\} $
is a sequence such that $v^{n}\in\partial\phi\left(\overline{y}-\sum_{i=1}^{n-1}\alpha_{i}\left\Vert v^{i}\right\Vert ^{-1}v^{i}\right)$
for each $n=1,2,\dots,p$. If, in addition, $\overline{y}\in B\left(y,2^{-1}r_{1}\right)$,
then
\[
\left\Vert \overline{y}-\sum_{n=1}^{p}\alpha_{n}\left\Vert v^{n}\right\Vert ^{-1}v^{n}-z\right\Vert ^{2}\le\left\Vert \overline{y}-z\right\Vert ^{2}-\sum_{n=1}^{p}\left(r_{2}-\alpha_{n}\right)\alpha_{n}
\]
(by definition $\sum_{n=1}^{0}\alpha_{n}\left\Vert v^{n}\right\Vert ^{-1}v^{n}:=\sum_{n=1}^{0}\left(r_{2}-\alpha_{n}\right)\alpha_{n}:=0)$.
\end{enumerate}
\end{lem}
The next theorem is ``a theorem of alternatives'' for the possible
outputs of Algorithm \ref{alg=0000202}. It generalizes Theorem 4.1
in \cite{CZ_sup} and Theorem 5.4 in \cite{GDSA_inconsist} to the
setting of Algorithm \ref{alg=0000202} which we apply in the next
section to the ``superiorized version of Algorithm \ref{alg=0000201}''.
\begin{thm}
\label{Strict=000020Fejer} Let $\phi:\mathcal{H}\rightarrow\mathbb{R}$
be a convex and continuous real valued objective function. Assume
that $\left\{ \mathcal{T}_{k}\right\} _{k=0}^{\infty}$ is a sequence
of nonexpansive operators having a common fixed point, where $\varepsilon\in\left(0,1\right]$.
Let $y^{0}\in\mathcal{H}$ and suppose that the sequence $\left\{ y^{k}\right\} _{k=0}^{\infty}$,
generated by Algorithm \ref{alg=0000202}, converges in the norm of
$\mathcal{H}$ to a point $y\in\cap_{k=0}^{\infty}\mathrm{Fix}\mathcal{T}_{k}$.
Then exactly one of the following two alternatives holds:
\begin{enumerate}
\item $y\in\mathrm{Argmin}_{x\in\cap_{k=0}^{\infty}\mathrm{Fix}\mathcal{T}_{k}}\phi\left(x\right)$.

\suspend{enumerate} ~~or \resume{enumerate}
\item $y\notin\mathrm{Argmin}_{x\in\cap_{k=0}^{\infty}\mathrm{Fix}\mathcal{T}_{k}}\phi\left(x\right)$
and there exists $k_{0}\in\mathbb{N}$ such that $\left\{ y^{k}\right\} _{k=k_{0}}^{\infty}$
is strictly Fej{\'e}r monotone with respect to $\mathrm{Argmin}_{x\in\cap_{k=0}^{\infty}\mathrm{Fix}\mathcal{T}_{k}}\phi\left(x\right)$,
that is, $\left\Vert y^{k+1}-z\right\Vert ^{2}<\left\Vert y^{k}-z\right\Vert ^{2}$
for every $z\in\mathrm{Argmin}_{x\in\cap_{n=0}^{\infty}\mathrm{Fix}U_{n}}\phi\left(x\right)$
and for all natural $k\ge k_{0}$.
\end{enumerate}
\end{thm}
\begin{proof}
Assume that $\left\{ y^{k}\right\} _{k=0}^{\infty}$ converges in
the norm of $\mathcal{H}$ to a point $y\not\in\mathrm{Argmin}_{x\in\cap_{k=0}^{\infty}\mathrm{Fix}\mathcal{T}_{k}}\phi\left(x\right)$
and that $z\in\mathrm{Argmin}_{x\in\cap_{k=0}^{\infty}\mathrm{Fix}\mathcal{T}_{k}}\phi\left(x\right)$.
By Lemma \ref{aux_lem}, since $y\in\cap_{k=0}^{\infty}\mathrm{Fix}\mathcal{T}_{k}$,
there exist real numbers $r_{1}>0$ and $r_{2}>0$ such that each
$\overline{y}\in B\left(y,r_{1}\right)$ and $v\in\partial\phi\left(\overline{y}\right)$
satisfy its assertions. By using the strong convergence of $\left\{ y^{k}\right\} _{k=0}^{\infty}$
to $y$ and the convergence of the series $\sum_{k=0}^{\infty}\sum_{n=1}^{M_{k}}\beta_{k,n}$,
choose $k_{0}\in\mathbb{N}$ such that 
\begin{equation}
y^{k}\in B\left(y,2^{-1}r_{1}\right)\label{eq:-30}
\end{equation}
 and 
\begin{equation}
\sum_{n=1}^{M_{k}}\beta_{k,n}<\min\left\{ 2^{-1}r_{1},r_{2}\right\} \label{eq:-14-1}
\end{equation}
for each integer $k\ge k_{0}$. This yields, for each $k\ge k_{0}$,
\[
y^{k}+\sum_{i=1}^{n-1}\beta_{k,i}v^{k,i}\in B\left(y,r_{1}\right)
\]
for each $n=1,2,\dots,M_{k}$, and, consequently, by Lemma \ref{aux_lem}\textit{(i)},
\begin{equation}
0\not\in\partial\phi\left(y^{k}+\sum_{i=1}^{n-1}\beta_{k,i}v^{k,i}\right)\label{eq:-17}
\end{equation}
for each $n=1,2,\dots,M_{k}$. Let $k\ge k_{0}$ be an integer. By
\eqref{eq:-19} and \eqref{eq:-17}, 
\begin{equation}
v^{k,n}=-\left\Vert s^{k,n-1}\right\Vert ^{-1}s^{k,n-1},\label{eq:-29}
\end{equation}
where 
\begin{equation}
s^{k,n-1}\in\partial\phi\left(y^{k}+\sum_{i=1}^{n-1}\beta_{k,i}v^{k,i}\right),\label{eq:-26}
\end{equation}
for each $n=1,2,\dots,M_{k}$. Define $p:=M_{k}$ and $\overline{y}:=y^{k}$.
For each $n=1,2,\dots,p$, define $\alpha_{n}:=\beta_{k,n}>0$ and
$v^{n}:=s^{k,n-1}$. Then, by \eqref{eq:-30}, \eqref{eq:-14-1},
\eqref{eq:-17}, \eqref{eq:-29} and \eqref{eq:-26}, $\overline{y}\in B\left(y,2^{-1}r_{1}\right)$,
$\sum_{n=1}^{p}\alpha_{n}<2^{-1}r_{1}$, $\left\{ v^{n}\right\} _{n=1}^{p}\subset\mathcal{H}\backslash\left\{ 0\right\} $
and $v^{n}\in\partial\phi\left(\overline{y}-\sum_{i=1}^{n-1}\alpha_{i}\left\Vert v^{i}\right\Vert ^{-1}v^{i}\right)$
for each $n=1,2,\dots,p$.

Since the operator $\mathcal{T}_{k}$ is nonexpansive, we obtain from
Lemma \ref{aux_lem}\textit{(iii),} that
\begin{align*}
\left\Vert y^{k+1}-z\right\Vert ^{2} & =\left\Vert \mathcal{T}_{k}\left(y^{k}+\sum_{n=1}^{M_{k}}\beta_{k,n}v^{k,n}\right)-z\right\Vert ^{2}\le\left\Vert y^{k}+\sum_{n=1}^{M_{k}}\beta_{k,n}v^{k,n}-z\right\Vert ^{2}\\
 & =\left\Vert \overline{y}-\sum_{n=1}^{p}\alpha_{n}\left\Vert v_{n}\right\Vert ^{-1}v_{n}-z\right\Vert ^{2}\le\left\Vert \overline{y}-z\right\Vert ^{2}-\sum_{n=1}^{p}\left(r_{2}-\alpha_{n}\right)\alpha_{n}\\
 & =\left\Vert y^{k}-z\right\Vert ^{2}-\sum_{n=1}^{M_{k}}\left(r_{2}-\beta_{k,n}\right)\beta_{k,n}.
\end{align*}
Since by \eqref{eq:-14-1}, $\sum_{n=1}^{M_{k}}\beta_{k,n}<r_{2}$,
we see that $\left\Vert y^{k+1}-z\right\Vert <\left\Vert y^{k+1}-z\right\Vert $
and the proof of the theorem is complete.
\end{proof}
\begin{rem}
The fundamental Algorithm \ref{alg=0000202} is a ``superiorization
algorithm'' meaning that it is not designed for reaching a constrained
optimal point. Its intentional aim is ``to do less than constrained
optimization'', namely, to enable to ``improve'' the feasibility-seeking
algorithm by steering it to an asymptotic feasible point whose objective
function value is smaller or equal to that of an asymptotic feasible
point that would have been reached by the same feasibility-seeking
algorithm under exact equal conditions without the perturbations.
The theorem-of-alternatives (Theorem \ref{Strict=000020Fejer}) is
the best result that we can establish at this time. It does not rule
out the possibility that Algorithm \ref{Strict=000020Fejer} will
reach a constrained optimal point (see Remark \ref{rem:ex} in the
next section in this connection) but tells that if this does not happen
then the generated sequence must be strictly Fej{\'e}r monotone.
It is probably possible to characterize when the superiorization algorithm
can solve a constrained optimization problem but this will come at
an unavoidable cost of imposing additional conditions. Such results
may be derived from, or be developed further from, the work done in
\cite{Combettes2000} and \cite{Combettes2024}. The leading motivation
for splitting methods, as described in \cite{Combettes2024}, is to
break a complex task to more manageable sub-tasks. In this sense the
superiorization can be viewed also as a splitting method. Its ties
to the broad literature on splitting methods have not yet been explored.
\end{rem}

\section{\protect\label{Apps}Applications to superiorization and Dynamic
String Averaging}

This section presents some consequences of the strong convergence
properties of our GMSA procedure for iterative methods based on the
GMSA procedure, and for the superiorization methodology, as developed
in the preceding sections.

We start by observing that under the assumptions of Theorem \ref{main_res}\textit{(iii)},
where, in particular, $\left\{ \lambda_{k}\right\} _{k=0}^{\infty}\subset\left[\varepsilon,1+\rho-\varepsilon\right]$
for $\varepsilon\in\left(0,1\right]$ and $\rho\in\left(0,1\right]$,
the sequence $\left\{ y^{k}\right\} _{k=0}^{\infty}$, generated by
Algorithm \ref{alg=0000202} with respect to the sequences $\left\{ \lambda_{k}\right\} _{k=0}^{\infty}$
and $\left\{ T_{k\lambda_{k}}\right\} _{k=0}^{\infty}$ (the $\lambda_{k}$-relaxations
of operators $\left\{ T_{k}\right\} _{k=0}^{\infty}$ generated by
the GMSA procedure obtained from the sequence of given input $2^{-1}$-strongly
quasi-nonexpansive operators $\left\{ U_{n}\right\} _{n=0}^{\infty}$
with nonempty intersection) converges in the norm of $\mathcal{H}$
to a point $y\in\cap_{n=0}^{\infty}\mathrm{Fix}U_{n}$.

Indeed, define a sequence $\left\{ \beta_{k}\right\} _{k=0}^{\infty}$
of positive real numbers by $0<\beta_{k}=\sum_{n=1}^{M_{k}}\beta_{k,n}$
for each $k\in\mathbb{N}$. Then $\sum_{k=0}^{\infty}\beta_{k}<\infty$
and we have, by defining the operators involved in Algorithm \ref{alg=0000202}
as $\mathcal{T}_{k}:=T_{k\lambda_{k}}$ for each $k\in\mathbb{N}$,
\begin{align}
y^{k+1} & =T_{k\lambda_{k}}\left(y^{k}+\sum_{n=1}^{M_{k}}\beta_{k,n}v^{k,n}\right)=T_{k\lambda_{k}}\left(y^{k}+\beta_{k}\sum_{n=1}^{M_{k}}\beta_{k,n}\beta_{k}^{-1}v^{k,n}\right),\label{eq:-7-1-2}
\end{align}
for each $k\in\mathbb{N}$. It follows, by the triangle inequality
and by \eqref{eq:-19}, that
\[
\left\Vert \sum_{n=1}^{M_{k}}\beta_{k,n}\beta_{k}^{-1}v^{k,n}\right\Vert \le\sum_{n=1}^{M_{k}}\beta_{k,n}\beta_{k}^{-1}=1,
\]
i.e., the real sequence $\left\{ \sum_{n=1}^{M_{k}}\beta_{k,n}\beta_{k}^{-1}v^{k,n}\right\} _{k=0}^{\infty}$
is bounded by 1 in $\mathcal{H}$. By Remark \ref{Relaxation=000020has=000020the=000020same=000020Fix}
and Lemma \ref{lem3}, we obtain in this case that 
\begin{equation}
\cap_{k=0}^{\infty}\mathrm{Fix}T_{k\lambda_{k}}=\cap_{k=0}^{\infty}\mathrm{Fix}T_{k}=\cap_{k=0}^{\infty}\cap_{n\in I_{N_{k}}^{k}}U_{n}=\cap_{n=0}^{\infty}U_{n}.\label{eq:-15}
\end{equation}
 By using \eqref{eq:-7-1-2}, Corollary \ref{cor:firm_nonex}, Lemma
\ref{outpt_nonexp}, Theorem \ref{Strict=000020Fejer} and \eqref{eq:-15},
we obtain the following corollary.
\begin{cor}
\label{SM_cor-1}Let $\phi:\mathcal{H}\rightarrow\mathbb{R}$ be a
convex and continuous real-valued objective function and let $\left\{ U_{n}\right\} _{n=0}^{\infty}$
be a sequence of approximately shrinking and $\gamma_{n}$-relaxed
firmly nonexpansive operators, where $\gamma_{n}\in\left(0,4\cdot3^{-1}\right]$
for each $n\in\mathbb{N}$, such that the family $\left\{ \mathrm{Fix}U_{n}\right\} _{n=0}^{\infty}$
is boundedly regular and $\cap_{n=0}^{\infty}\mathrm{Fix}U_{n}\not=\emptyset$.
Suppose that the sequence $\left\{ N_{k}\right\} _{k=0}^{\infty}$
is bounded and that for each $k\in\mathbb{N}$ and each $n\in c_{k}^{-1}\left(\left\{ 0\right\} \right)$,
in addition, the operator $U_{j_{n}^{k}}$ is firmly nonexpansive,
or that $\alpha_{k}\left(n\right)=1$. Let $\left\{ T_{k}\right\} _{k=0}^{\infty}$
be a sequence of operators generated by the GMSA procedure as in \eqref{eq:-3},
and assume that there exists $\rho\in\left[0,1\right]$ such that
the operators $T_{k}$ are $\rho$-firmly nonexpansive for all $k\in\mathbb{N}$
(by Corollary \ref{cor:firm_nonex}(ii), this holds in particular
when the sequence $\left\{ \left\{ P_{n}^{k}\right\} _{n\in\mathbb{N}\backslash\left\{ 0\right\} \cap L_{N_{k}}}\right\} _{k=0}^{\infty}$
is also bounded). Let $\varepsilon\in\left(0,1\right]$ and let $\left\{ \lambda_{k}\right\} _{k=0}^{\infty}\subset\left[1+\rho-\varepsilon\right]$
be a sequence. Assume that the sequence $\left\{ I_{N_{k}}^{k}\right\} _{k=0}^{\infty}$
satisfies Equation \eqref{eq:-16}. Then the sequence $\left\{ y^{k}\right\} _{k=0}^{\infty}$,
generated by Algorithm \ref{alg=0000202} with respect to the sequence
$\left\{ \lambda_{k}\right\} _{k=0}^{\infty}$ and $\left\{ T_{k\lambda_{k}}\right\} _{k=0}^{\infty}$
(the $\lambda_{k}$-relaxations of operators $\left\{ T_{k}\right\} _{k=0}^{\infty}$)
converges in the norm of $\mathcal{H}$ to a point $y\in\cap_{n=0}^{\infty}\mathrm{Fix}U_{n}$
and exactly one of the following two alternatives holds:
\begin{enumerate}
\item $y\in\mathrm{Argmin}_{x\in\cap_{n=0}^{\infty}\mathrm{Fix}U_{n}}\phi\left(x\right)$

\suspend{enumerate} ~~or \resume{enumerate}
\item $y\not\in\mathrm{Argmin}_{x\in\cap_{n=0}^{\infty}\mathrm{Fix}U_{n}}\phi\left(x\right)$
and there exists $k_{0}\in\mathbb{N}$ such that $\left\{ y^{k}\right\} _{k=k_{0}}^{\infty}$
is strictly Fej{\'e}r monotone with respect to $\mathrm{Argmin}_{x\in\cap_{n=0}^{\infty}\mathrm{Fix}U_{n}}\phi\left(x\right)$,
that is, $\left\Vert y^{k+1}-z\right\Vert ^{2}<\left\Vert y^{k}-z\right\Vert ^{2}$
for every $z\in\mathrm{Argmin}_{x\in\cap_{n=0}^{\infty}\mathrm{Fix}U_{n}}\phi\left(x\right)$
and for all natural $k\ge k_{0}$.
\end{enumerate}
\end{cor}
\begin{rem}
\label{rem:ex}Alternative 2 in Corollary \ref{SM_cor-1} is not merely
theoretical. Using the framework of Example \ref{MSA}, we can embed
concrete examples given, for instance, in \cite{neg_condit_2025}
into the setting of this alternative.
\end{rem}
A strongly convergent Dynamic String-Averaging Projection (DSAP) algorithm,
based on convex combinations and compositions of operators, was introduced,
along with its bounded perturbation resilience, by Censor and Zaslavski
in \cite{DSAP} for a family of nonempty, closed and convex sets $\left\{ C_{i}\right\} _{i=1}^{m}$,
where the considered operators were given by the metric projections
$\left\{ P_{C_{i}}\right\} _{i=1}^{m}$ in the consistent case, that
is when $\cap_{i=1}^{m}C_{i}\not=\emptyset$. A superiorized version
of this algorithm appears in \cite{CZ_sup}, where its properties
were studied.

String-averaging projection (SAP) methods form a general algorithmic
framework introduced in \cite{CEH2001}. Subsequently, they were developed
in a variety of situations such as for convex feasibility with infinitely
many sets \cite{Kong2019}, for incremental stochastic subgradient
algorithms \cite{costa} and for proton computed tomography image
reconstruction \cite{Barsik}, to name but a few. See also \cite{Bargetz2018},
where perturbation resilience of such methods was further studied.

In the recent paper by Barshad and Censor \cite{GDSA_inconsist} the
General Dynamic String-Averaging (GDSA) iterative scheme was presented
covering the case of empty intersection of fixed points sets of the
given operators (that is, the inconsistent case), wherein a finite
family of relaxed firmly nonexpansive operators was considered. Since
the emptiness of this intersection was assumed there, the conditions
which guarantee the strong convergence of the GDSA algorithm had to
be assumed on the output operators produced by it.

Using the above GMSA and SM machinery, we expand the GDSA method in
the consistent case to cover the case of infinitely many input operators
in the next example, where the aforesaid conditions are assumed on
these input operators.
\begin{example}
[The General Dynamic String-Averaging method in the consistent case]\label{GDSA}We
consider a sequence $\left\{ U_{n}\right\} _{n=0}^{\infty}$ of approximately
shrinking and $\gamma_{n}$-relaxed firmly nonexpansive operators
having a common fixed point, where $U_{n}:\mathcal{H\rightarrow\mathcal{H}}$
and $\gamma_{n}\in\left(0,4\cdot3^{-1}\right]$ for each $n\in\mathbb{N}$,
such that the family $\left\{ \mathrm{Fix}U_{n}\right\} _{n=0}^{\infty}$
is boundedly regular. Let $\left\{ q_{k}\right\} _{k=0}^{\infty}$
be a bounded sequence of positive integers, let $\left\{ J_{k}\right\} _{k=0}^{\infty}$
be a sequence of finite subsets of $\mathbb{N}$, whose cardinalities
are bounded by a positive integer $p$, and let $\left\{ \varOmega_{k}\right\} _{k=0}^{\infty}$
be a sequence of nonempty sets such that $\varOmega_{k}\subset J_{k}^{\left\{ 1,2,\dots,q_{k}\right\} }$
(that is, $\varOmega_{k}$ is a finite subset of the set of functions
from $\left\{ 1,2,\dots,q_{k}\right\} $ to $J_{k}$) for each $k\in\mathbb{N}$.
For each $k\in\mathbb{N}$ and each $t\in\varOmega_{k}$, define $V_{k}\left[t\right]:=U_{t\left(q_{k}\right)}\cdots U_{t\left(2\right)}U_{t\left(1\right)}$
and let $\omega_{k}:\varOmega_{k}\rightarrow\left[\varepsilon,1\right]$
(where $\varepsilon$ is a positive number) be a function such that
$\sum_{t\in\varOmega_{k}}\omega_{k}\left(t\right)=1$. For each $k\in\mathbb{N}$,
define $T_{\left(\varOmega_{k},\omega_{k}\right)}:=\sum_{t\in\varOmega_{k}}\omega_{k}\left(t\right)V_{k}\left[t\right]$.

Define $q:=\max\left\{ q_{k}\right\} _{k=0}^{\infty}$ and 
\[
\rho:=\min\left\{ q^{-1}\inf_{n\in\mathbb{N}}\left(2-\gamma_{n}\right)\gamma_{n},1\right\} \le1.
\]
\end{example}
For each $k\in\mathbb{N}$, we say that the set $\varOmega_{k}$ is
called ``fit'' if the sequence of sets $\left\{ I_{k}\right\} _{k=0}^{\infty}$,
defined by $I_{k}:=\cup_{t\in\varOmega_{k}}\im t$, for each $k\in\mathbb{N}$,
where $\im t$, which denotes the image of the mapping $t$, satisfies
Equation \eqref{eq:-16} for all $n\in\mathbb{N}$. For example, if
$f\left(k\right)\in\cup_{t\in\varOmega_{k}}\im t$ for each $k\in\mathbb{N}$,
where $\left\{ f_{k}\right\} _{k=0}^{\infty}$ is the sequence defined
in Example \ref{Concrete=000020example}, then the set $\varOmega_{k}$
is a fit for each $k\in\mathbb{N}$.

The operators $T_{\left(\varOmega_{k},\omega_{k}\right)}$ defined
above are ``string-averaging operators'' as introduced in \cite{CEH2001}
and further studied in various forms and settings, see, for instance,
\cite{GDSA_inconsist}, Example 5.21 in \cite{BC_book}, \cite{Kong2019}
and \cite{Nikazad2016}, to name but a few. In those and other papers,
the index vector $t$ is called a ``string'', the composite operator
$V_{k}\left[t\right]$ is called ``a string operator'' and $\omega_{k}$
are called ``weight functions''.

To properly embed this in the framework of Section \eqref{sec:The-strong-convergence}
we do as follows. Define a sequence $\left\{ N_{k}\right\} _{k=0}^{\infty}$
of positive integers by $N_{k}:=\left|\varOmega_{k}\right|+1$ for
each $k\in\mathbb{N}$. Clearly, $\left\{ N_{k}\right\} _{k=0}^{\infty}$
is bounded, since $N_{k}\le p^{q}+1$ for each $k\in\mathbb{N}$.
Let $t_{k}:\left\{ 1,2,\dots,\left|\varOmega_{k}\right|\right\} \rightarrow\varOmega_{k}$
be a bijection for each $k\in\mathbb{N}$. Now, for each $k\in\mathbb{N}$
and each positive integer $n\in L_{N_{k}-1}$, define $J_{n}^{k}:=-\im t_{k}\left(n\right)$,
where $\im t_{k}\left(n\right)$ is the image of the mapping $t_{k}\left(n\right)\in\varOmega_{k}$,
$c_{k}\left(n\right):=2$, $P_{n}^{k}:=q_{k}$ and $o_{n}^{k}:=-t_{k}\left(n\right)$.
Define also $J_{N_{k}}^{k}:=\left\{ 1,2,\dots,\left|\varOmega_{k}\right|\right\} $,
$c_{k}\left(N_{k}\right):=1$, and $P_{N_{k}}^{k}:=\left|\varOmega_{k}\right|$
for each $k\in\mathbb{N}$. Define $\omega_{n}^{k}:=\omega_{k}\left(t_{k}\left(n\right)\right)$
for each positive integer $n\in L_{N_{k}-1}$.

By \eqref{eq:}, 
\begin{equation}
I_{n}^{k}=\cup_{j\in J_{n}^{k}}I_{j}^{k}=-J_{n}^{k}=\im t_{k}\left(n\right)\label{eq:-18}
\end{equation}
for each $k\in\mathbb{N}$ and each positive integer $n\in L_{N_{k}-1}=L_{\left|\varOmega_{k}\right|}$.
By \eqref{eq:-18} and \eqref{eq:}, since $t_{k}$ is a bijection
for each $k\in\mathbb{N}$, we have 
\[
I_{N_{k}}^{k}=\cup_{j\in J_{N_{k}}^{k}}I_{j}^{k}=\cup_{j\in\left\{ 1,2,\dots,\left|\varOmega_{k}\right|\right\} }I_{j}^{k}=\cup_{j\in\left\{ 1,2,\dots,\left|\varOmega_{k}\right|\right\} }\im t_{k}\left(j\right)=\cup_{t\in\varOmega_{k}}\mathrm{Im}t
\]
for each $k\in\mathbb{N}$. Applying now the GMSA procedure with respect
to the sequence $\left\{ U_{n}\right\} _{n=0}^{\infty}$ we have,
for each $k\in\mathbb{N}$ and each positive $n\in L_{N_{k}-1}$,
\begin{equation}
V_{n}^{k}=V_{k}\left[t_{k}\left(n\right)\right]\,\,\mathrm{and}\,\,V_{N_{k}}^{k}=T_{\left(\varOmega_{k},\omega_{k}\right)}.\label{eq:-14}
\end{equation}

By Corollary \ref{comp_conv_comb_firm}\textit{(ii)}, for each $k\in\mathbb{N}$,
the operator $V_{k}\left(t\right)$ is $\rho$-firmly nonexpansive
for each $t\in\varOmega_{k}$. Therefore, by \ref{comp_conv_comb_firm},
$T_{\left(\varOmega_{k},\omega_{k}\right)}$ is $\rho$-firmly nonexpansive.
Thus, assuming that the set $\varOmega_{k}$ is a fit for each $k\in\mathbb{N}$,
we obtain the GDSA iterative scheme which fits the framework of Corollaries
\ref{cor:firm_nonex} and \ref{SM_cor-1}.

We use Example \ref{MSA} to embed in this scheme the GDSA method
in the consistent case for a given finite family $\left\{ U_{n}^{\prime}\right\} _{n=1}^{m}$
of input operators. By Examples \ref{metric=000020projection} and
\ref{ex_as}, we can choose, in particular, the input operators to
be metric projections and obtain the DSAP method with infinitely many
input operators, which was described for a finite number of these
operators in \cite{DSAP}.

\bigskip{}
\textbf{Acknowledgments}. We gratefully acknowledge the enlightening
and constructive comments of the three referees which helped us improve
the paper. This work is supported by U.S. National Institutes of Health
(NIH) Grant Number R01CA266467 and by the Cooperation Program in Cancer
Research of the German Cancer Research Center (DKFZ) and Israel's
Ministry of Innovation, Science and Technology (MOST).

\subsubsection*{Data availability}

No data was used for the research described in the article.

\bibliographystyle{j_style}
\bibliography{bank_of_references}

\end{document}